\documentclass[12pt]{article}
\usepackage{graphicx,color}
\usepackage{amsmath,amsfonts,amssymb,amsthm}
\usepackage{rsfso}
\usepackage{ulem}
\usepackage{mathrsfs}
\usepackage{enumerate}
\usepackage{algorithm}
\usepackage{algpseudocode}
\usepackage{comment}

\usepackage[numbers]{natbib}
\newcommand{\diag}[1]{{\rm diag}\left(#1\right)}

\newcommand{\rs}[1]{{\mbox{\scriptsize \sc #1}}}

\newcommand{\vc}[1]{{\boldsymbol #1}} 
\newcommand{\vcn}[1]{{\bf #1}}

\newcommand{\sr}[1]{{\mathcal #1}}

\newcommand{\dd}[1]{\mathbb{#1}}

\newcommand{\ex}{\rs{ex}}
\newcommand{\re}{\rs{re}}

\newcommand{\br}[1]{\langle #1 \rangle}
\newcommand{\ol}{\overline}

\newcommand{\eq}[1]{(\ref{eq:#1})}
\newcommand{\lem}[1]{Lemma~\ref{lem:#1}}

\newcommand{\thr}[1]{Theorem~\ref{thr:#1}}

\newcommand{\rem}[1]{Remark~\ref{rem:#1}}

\newcommand{\sectn}[1]{Section~\ref{sec:#1}}

\newcommand{\lemt}[1]{\ref{lem:#1}}

\newcommand{\thrt}[1]{\ref{thr:#1}}

\newcommand{\sect}[1]{\ref{sec:#1}}

\newcommand{\pend}{\hfill \thicklines \framebox(6.6,6.6)[l]{}}

\newenvironment{proof*}[1]{\noindent {\sc  #1} \rm}{\pend}

\newtheorem{theorem}{Theorem}[section]
\newtheorem{lemma}{Lemma}[section]

\newtheorem{remark}{Remark}[section]

\newcommand{\setnewcounter} {
\setcounter{subsection}{0}
\setcounter{equation}{0}
\setcounter{conjecture}{0}
\setcounter{assumption}{0}
\setcounter{question}{0}
\setcounter{definition}{0}
\setcounter{theorem}{0}
\setcounter{corollary}{0}
\setcounter{lemma}{0}
\setcounter{proposition}{0}
\setcounter{remark}{0}
}

\begin{document}
 \title{Markov modulated fluid network process: Tail asymptotics of the stationary distribution}

%%%%%%%%%%%%%%
%   AUTHORS  %
%%%%%%%%%%%%%%
\author{Masakiyo Miyazawa\\ Tokyo University of Science,\\ Chinese University of Hong Kong, Shenzhen}
\date{September 25, 2020, 2nd revised}

\maketitle

\begin{quote}
{\it The author congratulates Peter Taylor on the occasion of his birthday
 and thanks for his great contributions in research and acadenic service.}
\end{quote}

\begin{abstract}
We consider a Markov modulated fluid network with a finite number of stations. We are interested in the tail asymptotics behavior of the stationary distribution of its buffer content process. Using two different approaches, we derive upper and lower bounds for the stationary tail decay rate in various directions. Both approaches are based on a well-known time-evolution formula of a Markov process, so-called Dynkin's formula, where a key ingredient is a suitable choice of test functions. Those results show how multidimensional tail asymptotics can be studied for the more than two-dimensional case, which is known as a hard problem.

\bigskip

\noindent {\bf Keywords}: \  Fluid queue, network, Markov modulated, tail asymptotic behavior, Dynkin's formula, stationary inequality, change of measure, fixed point equation. 
\end{abstract}

\section{Introduction} \setnewcounter
\label{sec:Introduction}

A Markov modulated fluid queue with single buffer has been widely studied as a basic model in application, particularly related to quasi-birth-and-death process, QBD in short. Contrasted with them, their networks are hard to study even for the tail asymptotics of the stationary distribution. Because of this, feedforward networks such as tandem queues have been mainly studied (e.g., see \cite{AdanMandScheTzen2009,DebiDiekRols2007,Kell2001}). We attack the tail asymptotic problem for a Markov modulated fluid network with general routing topology.

This network model has $d$ stations for an arbitrary integer $d \ge 2$. Each station has a buffer with infinite capacity which has an exogenous fluid input and releases fluid as long as the buffer is not empty or has input fluid flow. A constant fraction of released fluid from one station is transferred to the other stations. Thus, this network has a general but deterministic routing. We assume that exogenous fluid input rates and release rates for nonempty buffers, which is called potential release rates, are determined by the current state of a continuous time Markov chain with finitely many states. This Markov chain is called a background process. Here, if a buffer is empty, then its release rate is the minimum of the input flow rate and the potential release rate. We refer to this Markov modulated fluid network as MMFN model, and describe it by a Markov process composed of a multidimensional reflecting fluid process and a background Markov chain. We call this Markov process an MMFN process.

Other than the MMFN process, there are different types of multidimensional reflecting processes depending on a net flow process. Their typical models for discrete state spaces are reflecting multidimensional random walks on nonnegative orthants of integer valued vector spaces and generalized Jackson networks (GJN for short), while semimartingale reflecting Brownian motions (SRBM for short) are typical for continuous state spaces. A Markov modulated fluid network process is somehow between them, and belong to a class of reflecting processes generated by multidimensional Markov additive processes. In those stochastic processes, difficulty for asymptotic analysis arises from reflection at the boundaries of their state spaces. For the two-dimensional case, the problem is relatively easy to study because the reflecting boundary is one-dimensional, that is, composed of two half lines.

During the last two decades, the tail asymptotic problems have been well studied for two-dimensional reflecting processes such as a random walk, GJN and SRBM (see, e.g., the survey paper \cite{Miya2011} and references in it). For the two-dimensional reflecting random walk, we must mention the pioneering work of Borovkov and Mogulski (e.g., \cite{BoroMogu2001}). The large deviation technique has been also applied for multidimensional reflecting processes under the name of Markov processes with discontinuous statistics (e.g., see \cite{AtarDupu1999,DupuElliWeis1991}), but results are less explicit for the tail asymptotic problems (see more discussions at the end of \sectn{asymptotic 2}).

Markov modulation naturally arises in those reflecting processes, and has been studied by matrix analysis. QBD processes are typical for it. For the network model, the level of a QBD process is multidimensional, and analysis is getting harder. There are some studies in the two-dimensional case (e.g, see \cite{Miya2015a,OzawKoba2018}). Their results may be sharp, but are generally very complicated because of lots of matrix operations. This also may come from complicated reflection to be allowed in the QBD setting. So far, it seems not easy to extend their results to the more than two-dimensional case. Because of this, we will not take the QBD approach for the present MMFN process.

In this paper, we only consider the tail decay rate, that is, logarithmic asymptotics, and provide two messages. First, a Markov modulation does not complicate the tail asymptotic problem compared with a multidimensional reflecting process without Markov modulation. We will show that it can be handed in the exactly same way as the latter process. One may wonder where the effect of Markov modulation goes in such an analysis. It is handled by matrix analysis, particularly, Perron Frobenius eigenvalue and eigenvector. This is our trick. Secondly, the more than two-dimensional case would be nicely handled by the solution of a fixed point equation for multidimensional sets (see \lem{fixed 1}). This idea may be applied to other types of a multidimensional reflecting process as well.

To bring those two messages, we first construct a sample path of a fluid network process as the solution of a multidimensional reflecting process, and define a Markov process by adding a background Markov chain. We then derive a well-known time-evolution formula of a Markov process, so-called Dynkin's formula, using a suitable test function (see \cite{EthiKurt1986} for Dynkin's formula). This is the starting point of our analysis. We next assume the stability condition for the MMFN process, which is well known (e.g., see \cite{KellWhit1996}). To study the tail asymptotics of the stationary distribution of the buffer content, we employ two different approaches. As we detail below, both approaches have been studies for other types of multidimensional reflecting processes.

One is the stationary equation and inequalities in terms of a certain type of moment generating functions. We will use the Dynkin's formula for their derivations, but one can do it directly using the generator of a Markov process. A merit using the Dynkin's formulas is to simplify their derivations through flexibly choosing test functions. The same idea has been used for a two-dimensional SRBM in \cite{DaiMiya2011}, although there is no background process involved. The stationary equation of the SRBM is called a basic adjoint relation, BAR for short. Our stationary equation includes the background state distributions, but we will show that, if test functions are suitably chosen, then the stationary equation and inequalities can be used in a similar way to the BAR of the SRBM for the more than two-dimensional case (see \lem{marginal 2} and \thr{upper 1}).

Another approach is to change the probability measure by an exponential martingale also obtained from the Dynkin's formula. This approach has been used for the GJN in \cite{Miya2017a}. It presents the tail probability of the stationary distribution under a new measure, which makes the asymptotic analysis tractable because the Markov process may diverge under the new measure. In principle, this approach is applicable to any shape of the tail set unlike the BAR approach, which is based on moment generating functions and therefore the shape is limited to an envelope by hyperplanes. The cost for this wide applicability is demand for knowledge about the technic for martingale and change of measure. However, it turns out that the merit of this high cost is rather small in the present work (see \thr{lower 1}). We will see that both approaches produce almost the same results about upper bounds for the tail decay rates, while the change of measure approach has some merits for the lower bounds. Nevertheless, we think that this approach is worthwhile to study because it elucidates dynamics of the tail probability with help of the new measure.

This paper is made up by seven sections. In \sectn{MMFN}, the MMFN (Markov modulated fluid network) process is introduced, which is a piece-wise linear Markov process, and Dynkin's formula is derived using exponential type of test functions. In \sectn{main}, main results of this paper, Theorems \thrt{upper 1} and \thrt{lower 1}, are presented, which give bounds for the tail decay rates. In \sectn{upper}, the existence of the solution for the fixed point equation is proved using stationary inequalities, by which \thr{upper 1} is proved. In \sectn{lower}, \thr{lower 1} is proved, deriving the stationary tail probability under change of measure. Most of lemmas, not proved in the main text, are proved in \sectn{proof}. Finally, \sectn{concluding} remarks about the present and future studies.

\section{Markov modulated fluid network} \setnewcounter
\label{sec:MMFN}

In this section, we first formally introduce a Markov modulated fluid network, and show that its buffer content process is a reflecting process which is characterized by fluid flow equations through so-called Skorohod map. We then construct a Markov process composed of this reflecting fluid process and a background Markov chain. For this Markov process, we derive a Dynkin's formula, which is our main tool.

\subsection{Modeling assumptions and dynamics}
\label{sec:dynamics}

For a formal description of a fluid network model, we introduce modeling primitives and assumptions for them. This network has $d$ stations numbered as $1,2,\ldots, d$ for an integer $d \ge 2$. Let $\sr{K} = \{1,2,\ldots,d\}$. We assume that fluid input and potential release rates are controlled by a continuous time Markov chain  with a discrete state space $\sr{J}$, which is called a background process. Denote this Markov chain and its transition matrix by $\{J(t); t \ge 0\}$ and $Q \equiv \{q_{ij}; i,j \in \sr{J}\}$ respectively, where $q_{ii} = - \sum_{j \in \sr{J} \setminus \{i\}} q_{ij}$.  We assume
\begin{itemize}
\item [(\sect{MMFN}a)] $\sr{J} \equiv \{1,2,\ldots, m\}$ for a finite integer $m \ge 2$, and $Q$ is irreducible.
\end{itemize}

For background state $i \in \sr{J}$ and station $k \in \sr{K}$, denote the exogenous fluid input and potential release rates at station $k$ by $\lambda_{k}(i)$ and $\mu_{k}(i)$, respectively. If no fluid is buffered at station $k$, the actual release rate is the minimum of the total inflow rate and $\mu_{k}(i)$. In any case, the fraction $p_{k,\ell}$ of the released fluid at station $k$ goes to station $\ell$. Thus, the fraction of the fluid leaving the system is $p_{k,0} \equiv 1 - \sum_{\ell \in \sr{K}} p_{k,\ell}$, and the routing mechanism does not depend on the background process ${J(t)}$. We may assume that $p_{k,k} = 0$ for $k \in \sr{K}$, but it is not needed for our arguments, so we do not assume it.

Let $P = \{p_{k,\ell}; k,\ell \in \sr{K}\}$. Throughout the paper, in addition to (\sect{MMFN}a), we assume
\begin{itemize}
\item [(\sect{MMFN}b)] Let $\ol{P}$ be the $(d+1)$-dimensional matrix obtained by adding the $0$-th column $\{p_{k,0}; k=0,1,\ldots,d \}$ and the $0$-th row $\{p_{0,\ell}; \ell = 1,2,\ldots, d \}$ to $P$, where $p_{0,0} = 0$ and $p_{0,\ell} = \sum_{i \in \sr{J}} \lambda_{k}(i)/\sum_{i \in \sr{J}} \sum_{k \in \sr{K}} \lambda_{k}(i)$. Then, $\ol{P}$ is irreducible.
\end{itemize}
Note that condition (\sect{MMFN}b) means that the network can not be separated as disjoint two or more networks. It also implies that $\lim_{n \to \infty} P^{n} = 0$, which shows that the fluid eventually flows out if the potential release rates are sufficiently large. We refer to this fluid flow model as a Markov modulated fluid network, MMFN model for short.

In this paper, we use the following notations. The sets of all real numbers and of all nonnegative real numbers, respectively, are denoted by $\dd{R}$ and $\dd{R}_{+}$, and the complex number field is denoted by $\dd{C}$. For vector and matrix notations, the following conventions are used. Vectors are of column unless stated otherwise. The unit vector whose $k$-th entry is $1$ while all the other entries vanish is denoted by $\vcn{e}_{k}$. For vectors $\vc{x} \equiv (x_{1}, x_{2}, \ldots, x_{d})^{\rs{t}}, \vc{y} \equiv (y_{1}, y_{2}, \ldots, y_{d})^{\rs{t}}$ in $\dd{R}^{d}$ (or $\dd{C}^{d}$), where ``$\rs{t}$'' indicates the transpose of a vector, their inner product, $\sum_{i \in \sr{K}} x_{i} y_{i}$, is denoted by $\br{\vc{x}, \vc{y}}$. For $d$-dimensional vector $\vc{x} \in \dd{R}^{d}$ and set $A \subset \sr{K}$, $\vc{x}_{A}$ is the $|A|$-dimensional vector whose $i$-th entry is $x_{i}$ for $i \in A$, where $|A|$ is the number of elements of $A$. For $d$-dimensional square matrix $T \equiv \{t_{ij}; i,j \in \sr{K}\}$ and sets $A, B \subset \sr{K}$, $T_{A,B}$ is the $|A| \times |B|$-matrix whose $(i,j)$ entry is $t_{i,j}$ for $i \in A, j \in B$.

As for state space $S \equiv \dd{R}^{d}$, we define $S_{k}$ for $k \in \sr{K}$ and $S_{A}$ for $A \subset \sr{K}$ and $A \ne \emptyset$ as
\begin{align*}
  S_{k} = \{\vc{z} \in S; z_{k} = 0\}, \qquad S_{A} = \bigcup_{k \in A} S_{k}, \qquad \vc{z} \equiv (z_{1},\ldots,z_{d})^{\rs{t}}, 
\end{align*}
For $A = \emptyset$, we let $S_{\emptyset} = S_{+} \equiv \{\vc{z} \in S; z_{k} > 0, \forall k \in \sr{K}\}$, which is the inside of $S$, and let $\partial S = S_{\sr{K}}$, which is the boundary of $S$. For $A \ne \emptyset$, we refer to $S_{A}$ as a face $A$. This $S_{A}$ should not be confused with the set $\{\vc{z} \in S; z_{k} = 0, \forall k \in A, z_{k} > 0, \forall k \in \sr{K} \setminus A\}$, which is denoted by $S_{A}$ in \cite{Miya2011}.

\subsection{Buffer content and reflecting fluid network processes}
\label{sec:sample}

We first focus on the buffer content process of the MMFN model, which will be generated from a net flow process, introduced below. A sample path of the buffer content process is a function from $\dd{R}_{+}$ to $S$. In what follows, we assume that all the continuous time processes are right continuous and have left-hand limits. For them, the following sets of functions are convenient. Denote the sets of all continuous functions and all right-continuous functions with left-limits, respectively, by $C(S)$ and $D(S)$. For $f \in C(S)$ or $f \in D(S)$, define the uniform norm by
\begin{align*}
  \|f\|_{t} = \sup_{0 \le u \le t} |f(u)|, \quad t > 0, \qquad \|f\| = \sum_{n=1}^{\infty} 2^{-n} \|f\|_{n}.
\end{align*}
With this norm, $D(S)$ is a complete metric space but not separable, while $C(S)$ is a complete and separable metric space (see, e.g., Section 3 of \cite{Whit2002}). Similarly, we define function spaces $C(\dd{R}^{d})$ and $D(\dd{R}^{d})$ with uniform norm for $\dd{R}^{d}$ instead of $S$.

For $k \in \sr{K}$, let
\begin{align}
\label{eq:v k}
 & v_{k}(i) = \lambda_{k}(i) + \sum_{\ell \in \sr{K}} \mu_{\ell}(i) p_{\ell,k} - \mu_{k}(i), \qquad i \in \sr{J},\\
\label{eq:U k}
 & V_{k}(t) = \int_{0}^{t} v_{k}(J(s)) ds,
\end{align}
and denote the vector whose $k$-th entry is $V_{k}(t)$ by $\vc{V}(t)$. We call $\vc{V}(\cdot) \equiv \{\vc{V}(t); t \ge 0\}$ as a net flow process because it represents the difference of the inflow and outflow when all the buffers are not empty. Obviously, a sample path of $\vc{V}(\cdot)$ is in $C(\dd{R}^{d})$.

For each station $k \in \sr{K}$, let $Z_{k}(t)$ be the buffer content at time $t \ge 0$, and let $B_{k}(t)$ is the total amount of fluid released up to time $t$. Let $\vc{Z}(t)$ and $\vc{B}(t)$ be the $d$-dimensional vectors whose $k$-th entries are $Z_{k}(t)$ and $B_{k}(t)$ respectively. To formally define the sample path of $\{(\vc{Z}(t),\vc{B}(t)); t \ge 0\}$, we assume the following conditions. First, it must satisfy the flow balance equation:
\begin{align}
\label{eq:Z 1}
  Z_{k}(t) = Z_{k}(0) + \int_{0}^{t} \lambda_{k}(J(s)) ds + \sum_{\ell \in \sr{K}} \int_{0}^{t} p_{\ell k} B_{\ell}(ds) - B_{k}(t) \ge 0,
\end{align}
We next assume the following condition. 
\begin{align}
\label{eq:differentiable 1}
\begin{array}{ll}
 & \mbox{For all $k \in \sr{K}$, $B_{k}(t)$ has a derivative except for finitely many points }\\
 & \mbox{in each finite time interval.} 
\end{array}
\end{align}
By this and \eq{Z 1}, $Z_{k}(t)$ also has a derivative when $B_{k}(t)$ does so. Denote the derivatives of $Z_{k}(t)$ and $B_{k}(t)$ by $z_{k}(t)$ and $b_{k}(t)$ when they exist. Then, the following conditions should be satisfied from our modeling assumptions.
\begin{align}
\label{eq:b 1}
 & b_{k}(t) \le \mu_{k}(J(t)), \quad b_{k}(t) = \mu_{k}(J(t)) \mbox { if } Z_{k}(t)>0,\\
\label{eq:b 2}
 & z_{k}(t) = 0 \mbox { if } Z_{k}(t) = 0.
\end{align}

These assumptions yield the following fact, which is proved in \sectn{b 1}.
\begin{lemma}\rm
\label{lem:b 1}
If $\vc{Z}(\cdot)$ and $\vc{B}(\cdot)$ satisfy \eq{Z 1}--\eq{b 2}, then we have
\begin{align}
\label{eq:b 3}
  b_{k}(t) & = \mu_{k}(J(t))1(Z_{k}(t)>0) + (\mu_{k}(J(t)) \wedge a_{k}(t)) 1(Z_{k}(t)=0) \ge 0, \quad k \in \sr{K},
\end{align}
for $t$ not in the exceptional times, where $a_{k}(t) = \lambda_{k}(J(t)) + \sum_{\ell \in \sr{K}} b_{\ell}(t) p_{\ell,k}$, and $x \wedge y = \min(x,y)$ for $x,y \in \dd{R}$.
\end{lemma}
\begin{remark}\rm
\label{rem:b 1}
This lemma is intuitively clear if we interpret $a_{k}(t)$ as the total arrival rate and $b_{k}(t)$ as the departure rate at station $k$ at time $t$. However, we do need a proof because we do not assume that $b_{k}(t) \ge 0$ and any specific form for $b_{k}(t)$ for $Z_{k}(t)=0$, which are not immediate from \eq{Z 1} and \eq{b 1}.
\end{remark}

Define matrix $R$ as
\begin{align}
\label{eq:R 1}
  R = I - P^{\rs{t}},
\end{align}
where the $(k,\ell)$-entry of $R$ is denoted $r_{k,\ell}$, and define $\vc{Y}(\cdot) \equiv \{\vc{Y}(t); t \ge 0\}$ by
\begin{align}
\label{eq:Y 1}
  Y_{k}(t) =\int_{0}^{t} \mu_{k}(J(s)) ds - B_{k}(t), \qquad t \ge 0,
\end{align}
then then \eq{Z 1} can be written as
\begin{align}
\label{eq:Z 2}
  \vc{Z}(t) = \vc{Z}(0) + \vc{V}(t) + R \vc{Y}(t) \ge \vc{0}, \qquad t \ge 0,
\end{align}
and \lem{b 1} implies that $Y_{k}(t)$ increases only when $Z_{k}(t) = 0$, and therefore
\begin{align}
\label{eq:ZY 1}
 \int_{0}^{t} Z_{k}(s) Y_{k}(ds) = 0, \quad Y_{k}(0) = 0, \quad \mbox{$Y_{k}(t)$ is nondecreasing in $t$}, \quad k \in \sr{K}, t \ge 0.
\end{align}

Thus, we arrive at the standard formulation of the reflecting process with reflection matrix $R$. Since $R$ is an $\sr{M}$-matrix, for each sample path, the solution $\Psi(\vc{V}(\cdot)) \equiv (\vc{Z}(\cdot),\vc{Y}(\cdot)) \in C(S) \times C(S)$ of \eq{Z 2} and \eq{ZY 1} uniquely exists for $\vc{V}(\cdot) \in C(\dd{R}^{d})$ and is Lipschitz continuous  under the uniform norm generated from those on compact sets (see, e.g., Theorem 7.2 of \cite{ChenYao2001}). This solution $\Psi$ is called a Skorohod map. Thus, $(\vc{Z}(\cdot),\vc{B}(\cdot))$ in our original setting is given by $(\vc{Z}(\cdot),\vc{Y}(\cdot))$ if we define $\vc{B}(\cdot)$ through \eq{Y 1}. This proves the existence of $(\vc{Z}(\cdot),\vc{B}(\cdot))$. We finally show that  \eq{differentiable 1}, \eq{b 1} and \eq{b 2} are not extra conditions for the formulations by \eq{Z 2} and \eq{ZY 1}. First, \eq{differentiable 1} is immediate from Lipschitz continuity of $\Psi(\vc{V}(\cdot))$. Since $\vc{V}(\cdot)$ and $\vc{Z}(\cdot)$ are piecewise deterministic, where the latter is checked by inspecting the sample path of $\vc{Z}(\cdot)$, and \eq{Z 2} implies that
\begin{align*}
  \vc{Y}(t) = R^{-1}(\vc{Z}(t) - \vc{Z}(0) - \vc{V}(t)),
\end{align*}
\eq{b 1} and \eq{b 2} are also immediate from \eq{differentiable 1}, \eq{Y 1} and \eq{ZY 1}. Thus, we have obtained the following lemma.
\begin{lemma}\rm
\label{lem:reflection map 1}
$\vc{Z}(\cdot)$ and $\vc{B}(\cdot)$ satisfying \eq{Z 1}--\eq{b 2} uniquely exist, and are Lipschitz continuous on each time interval $[0,t]$ for $t > 0$.  
\end{lemma}
\begin{remark}\rm
\label{rem:reflection map 1}
(a) This lemma holds as long as each sample path of $V_{k}(t)$ is differentiable except for finitely many points in each finite time interval for each $k \in \sr{K}$. Thus, Markov modulation is not essential here.\\
(b) $Y_{k}(t)$ has the derivative $y_{k}(t)$ corresponds to the fixed point equation for $\pi(y)$ of \cite[(7.20) on page 166]{ChenYao2001}. This suggests that there is a derivative version of the Skorohod solution for \eq{Z 2} and \eq{ZY 1} if $\vc{V}(\cdot)$ is differentiable.
\end{remark}

By \lem{reflection map 1}, $\vc{Z}(\cdot) \equiv \{\vc{Z}(t); t \ge 0\}$ and $\vc{B}(\cdot) \equiv \{\vc{B}(t); t \ge 0\}$ are formally defined as the solution of \eq{Z 1}--\eq{b 2}, where $\{\lambda_{k}(J(t)), \mu_{k}(J(t)); k \in \sr{K}\}$ and the routing matrix $P$ are given as modeling primitives of the MMFN model.

\subsection{Markov process for the MMFN model and stability}
\label{sec:stability}

We now describe the reflecting fluid network satisfying the conditions (\sect{MMFN}a) and (\sect{MMFN}b) by a Markov process. For this, we take the joint process $\vc{X}(t) \equiv (\vc{Z}(\cdot),J(\cdot))$ defined on an appropriately chosen probability space $(\Omega,\sr{F},\dd{P})$, where $J(\cdot) \equiv \{J(t); t \ge 0\}$ is a continuous time Markov chain with state space $\sr{J} \equiv \{1,2,\ldots,m\}$ and transition rate matrix $Q$. $J(\cdot)$ is called a background process. Then, it is easy to see that $\vc{X}(\cdot) \equiv \{\vc{X}(t); t \ge 0\}$ is a continuous time Markov process with state space $S \times \sr{J}$ adapted to filtration $\dd{F} \equiv \{\sr{F}_{t}; t \ge 0\}$, where $\sr{F}_{t} = \sigma(\{\vc{X}(u); 0 \le u \le t\})$, and $\sr{F}$ is chosen so that it includes $\cup_{t \ge 0} \sr{F}_{t}$. In what follows, we omit $\dd{F}$ as long as it is easily identified in the context. Since this Markov process describes the Markov modulated fluid network, we refer to it as a Markov modulated fluid network process, MMFN process for short. The modeling parameters of this $d$-dimensional reflecting process are given by $\{v_{k}(j); k \in \sr{K}, j \in \sr{J}\}$, $d$-dimensional reflecting matrix $R$ and $m$-dimensional transition matrix $Q$.

In this subsection, we consider the stability of the Markov process $\vc{X}(\cdot)$, that is, the existence of its stationary distribution. Obviously, the mean drift of the net flow process $\vc{V}(\cdot)$ is a key for this. So, we introduce the following notation. By (\sect{MMFN}a), $J(t)$ has a unique stationary distribution, which is denoted by $\pi \equiv \{\pi(i); i \in \sr{J}\}$. Let
\begin{align}
\label{eq:stability}
  \ol{v}_{k} = \sum_{i \in \sr{J}} v_{k}(i) \pi(i), \qquad k \in \sr{K}.
\end{align}
$\ol{v}_{k}$ represents the mean drift in the $k$-th coordinate direction. Let $\ol{\vc{v}} = (\ol{v}_{1}, \ldots, \ol{v}_{d})^{\rs{T}}$. 

We first refer to the stability of a stochastic fluid network obtained by \citet{KellWhit1996}. They consider a general net flow process $\vc{V}(\cdot)$ in our notation but all the other conditions are the same.
\begin{lemma}[Theorem 10 and its corollary of \citet{KellWhit1996}]
\label{lem:stability 1}
For the reflecting process $\vc{Z}(\cdot)$ characterized by \eq{Z 2} and \eq{ZY 1}, if a net flow process $\vc{V}(\cdot)$ has stationary increments and if
\begin{align}
\label{eq:stability 1}
 & R^{-1} \ol{\vc{v}} < \vc{0},
\end{align}
then $\vc{Z}(\cdot)$ is tight, that is, for any $\varepsilon > 0$, there exists a compact set $C$ of $S$ such that $\dd{P}(\vc{Z}(t) \in C) \ge 1 - \varepsilon$ for all $t \ge 0$, for any proper distribution of $\vc{Z}(0)$. In particular, if $\vc{Z}(0)=\vc{0}$, then $\vc{Z}(t)$ converges in distribution to a proper limit as $t \to \infty$.
\end{lemma}

The $\vc{V}(\cdot)$ of this paper does not have stationary increments unless $J(0)$ is subject to $\pi$, but it can be coupled with the one which has stationary increments in a finite time w.p.1 because $J(\cdot)$ is a Markov chain. Namely, denote the $J(\cdot)$ which has the initial distribution $\pi$ by $\widetilde{J}(\cdot)$, then there is a finite nonnegative random variable $\tau$ such that $J(t)$ and $\widetilde{J}(t)$ are stochastically identical for $t \ge \tau$, which is well known for a discrete time Markov chain (e.g. see Proposition 3.13 of \cite{Asmu2003}) and its extension to the continuous time case is obvious. Then, $V(t)$ has stationary increments after time $\tau$ if $J(t)$ is replaced by $\widetilde{J}(t)$ after that time. Thus, all the results of \lem{stability 1} are valid for the present net flow process $\vc{V}(\cdot)$, and we immediately see the following fact.

\begin{lemma} \rm
\label{lem:stability 2}
The fluid network process $\vc{X}(\cdot)$ satisfying the conditions (\sect{MMFN}a) and (\sect{MMFN}d), that is, the MMFN process, has a stationary distribution if \eq{stability 1} holds.
\end{lemma}

We next interpret the stability condition \eq{stability 1} in terms of in and out flow rates at stations. Namely, let $\{\alpha_{k}(i); k \in \sr{K}, i \in \sr{J}\}$ be the solution of the following linear traffic equation.
\begin{align}
\label{eq:traffic 1}
  \alpha_{k}(i) = \lambda_{k}(i) + \sum_{\ell \in \sr{K}} \alpha_{\ell}(i) p_{\ell,k}, \qquad k \in \sr{K}, i \in \sr{J}.
\end{align}
Recall that $\pi$ is the stationary distribution of the background process $J(\cdot)$, and let
\begin{align*}
  \ol{\alpha}_{k} = \sum_{i \in \sr{J}} \pi(i) \alpha_{k}(i), \qquad \ol{\lambda}_{k} = \sum_{i \in \sr{J}}  \pi(i) \lambda_{k}(i), \qquad \ol{\mu}_{k} = \sum_{i \in \sr{J}}  \pi(i) \mu_{k}(i), \qquad k \in \sr{K}.
\end{align*}
Then, \eq{traffic 1} can be written as $\ol{\vc{\alpha}} = R^{-1} \ol{\vc{\lambda}}$, where $\ol{\vc{\alpha}}$ and $\ol{\vc{\lambda}}$ are the column vectors whose $k$-th entries are $\ol{\alpha}_{k}$ and $\ol{\lambda}_{k}$, respectively. On the other hand, it follows from \eq{v k} that
\begin{align*}
  R^{-1} \ol{\vc{v}} = R^{-1} (\ol{\vc{\lambda}} - R \ol{\vc{\mu}}) = \ol{\vc{\alpha}} - \ol{\vc{\mu}}.
\end{align*}
and therefore the stability condition \eq{stability 1} is equivalent to the standard condition:
\begin{align}
\label{eq:stability 2}
  \ol{\vc{\alpha}} < \ol{\vc{\mu}}.
\end{align}

Note that \lem{stability 2} tells nothing about the uniqueness of the stationary distribution, although it likely holds. This does not cause any problem in our analysis because our results are valid for any stationary distribution. Although we do not need the uniqueness of the stationary distribution, it would be nice if it can be proved. In \sectn{tail}, we give some related result as \lem{reachable 1}. In general, it requires to study a certain irreducibility for the Markov process $\vc{X}(\cdot)$ such as $\psi$-irreducibility (see, e.g., \cite{MeynTwee1993}). However, such an irreducibility is not easy to get because $v_{k}(J(t))$ and $v_{\ell}(J(t))$ are not independent for $k \ne \ell$. This is contrasted with $\psi$-irreducibility for the generalized Jackson network (e.g., see \cite{Dai1995,MeynDown1994}). So far, we will not further consider the uniqueness of the stationary distribution.

In our arguments, the MMFN process will be studied also under change of measure. In this case, the stability condition \eq{stability 2} may not hold. Thus, we need to consider the situation that there is a station $k \in \sr{K}$ such that $\ol{\alpha}_{k} \ge \ol{\mu}_{k}$. In particular, we obviously have the following fact by considering a fluid limit (see \sectn{ZA limit 1} for fluid scaling limit).

\begin{lemma}\rm
\label{lem:unstable 1}
For $k \in \sr{K}$, if $\ol{\alpha}_{k} > \ol{\mu}_{k}$, equivalently, $[R^{-1} \ol{\vc{v}}]_{k} > 0$ and if all the other stations are stable, then station $k$ is unstable, that is, $Z_{k}(t)$ has no stationary distribution.
\end{lemma}
It should be noted that, if there exists more than one station $k$ such that $\ol{\alpha}_{k} > \ol{\mu}_{k}$, all of them may not be unstable because the departure rates of those stations are bounded by the release rates and therefore they may be smaller than the total arrival rates obtained as the solutions of the traffic equation \eq{traffic 1}.

To find unstable stations, we need to solve the nonlinear traffic equations:
\begin{align}
\label{eq:traffic 2}
  \alpha_{k}(i) = \lambda_{k}(i) + \sum_{\ell \in \sr{K}} \min(\alpha_{\ell}(i), \mu_{\ell}(i)) p_{\ell,k}, \qquad k \in \sr{K}, i \in \sr{J}.
\end{align}
It is known that this traffic equation has a maximal solution for each $i \in \sr{J}$ (see Theorem 3.1 of \cite{ChenMand1991}). Denote this solution $\alpha_{k}(i)$ by $\alpha^{*}_{k}(i)$ and its average under $\pi$ by $\ol{\alpha}^{*}_{k}$. Then, ignoring the stations for which the net flow rates are nonnegative while keeping their out flows and applying the arguments of \cite{KellWhit1996}, we can see that station $k \in \sr{K}$ is stable if 
\begin{align}
\label{eq:stability 0}
  \ol{\alpha}^{*}_{k} < \ol{\mu}_{k}.
\end{align}
Although \eq{traffic 2} can be solved in a finite number of arithmetic operations, the solution is analytically intractable.  So, we will not solve \eq{traffic 2}, but use the fact that $\ol{v}_{k} < 0$ implies \eq{stability 0} under a change of measure (see \lem{stability 3}).\vspace{-1ex}

\subsection{Dynkin's formula for the MMFN process}
\label{sec:Dynkin}

In this subsection, we derive a time evolution formula for the MMFN process known as a Dynkin's formula, which is a semi-martingale with an absolutely continuous bounded variation component. This representation will be used to study the tail decay rates of the stationary distribution of $\vc{Z}(t)$. Let us introduce some notations. $C^{1}(\dd{R}^{d})$ denotes the subset of $C(\dd{R}^{d})$ whose elements are continuous partial derivatives, and $F_{+}(\sr{J})$ denotes the set of all functions from $\sr{J}$ to $\dd{R}_{+} \setminus \{0\}$.

For $n \ge 1$, let $t_{n}$ be the $n$-th transition instants of the background Markov chain $J(\cdot)$, and let $N(t) = \sup\{n \ge 0; t_{n} \le t\}$ for $t \ge 0$, where $t_{0} = 0$. Then, $N(\cdot) \equiv \{N(t); t \ge 0\}$ is a simple counting process, that is, $N(\cdot)$ is a nondecreasing piecewise constant process with unit jumps. For $g \in C^{1}(\dd{R}^{d})$ and $h \in F_{+}(\sr{J})$, let $f(\vc{z},j) = g(\vc{z}) h(j)$ for $(\vc{z},j) \in S \times \sr{J}$. Then, recalling that $\vc{X}(t) = (\vc{Z}(t), J(t))$ and using an elementary integration formula that
\begin{align*}
  f(\vc{X}(t)) = f(\vc{X}(0)) + \int_{0}^{t} d f(\vc{X}(u)), \qquad t \ge 0,
\end{align*}
\eq{Z 2} and the differentiability of $Z_{k}(t), V_{k}(t), Y_{k}(t)$ yield
\begin{align}
\label{eq:evolution 1}
 & f(\vc{X}(t)) = f(\vc{X}(0)) + \int_{0}^{t} \sum_{k = 1}^{d} g_{k}(\vc{Z}(s)) h(J(s)) v_{k}(J(s)) ds \nonumber\\
  & \qquad +\int_{0}^{t} \sum_{k, \ell \in \sr{K}} g_{k}(\vc{Z}(s)) h(J(s)) r_{k,\ell} y_{\ell}(s) ds + \int_{0}^{t} g(\vc{Z}(s)) \Delta h(J(s)) dN(s), \; t \ge 0,
\end{align}
where recall that $r_{k,\ell}$ is the $(k,\ell)$ entry of $R$, $g_{k}(\vc{z}) = \frac {\partial}{\partial z_{k}} g(\vc{z})$, and $\Delta h(J(s)) = h(J(s)) - h(J(s-))$. Then, it is not hard to see the following formulas, called Dynkin's formula. For completeness of our arguments, we give its proof in \sectn{Dynkin 1}.

\begin{lemma}\rm
\label{lem:Dynkin 1}
For $f(\vc{z},j) = g(\vc{z}) h(j)$ for $(\vc{z},j) \in S \times \sr{J}$ with $g \in C^{1}(\dd{R}^{d})$ and $h \in F_{+}(\sr{J})$, $f(\vc{X}(t))$ is a semi-martingale such that
\begin{align}
\label{eq:Dynkin 1}
  f(\vc{X}(t)) & = f(\vc{X}(0)) + \int_{0}^{t} \sr{A}f(\vc{X}(s)) ds + M(t), \qquad t \ge 0,
\end{align}
where $M(\cdot) \equiv \{M(t); t \ge 0\}$ is an $\dd{F}$-martingale, and $\sr{A} f$ is the generator of the Markov process $\vc{X}(\cdot)$ defined as
\begin{align}
\label{eq:A 1}
 \sr{A}f(\vc{X}(t)) & =  \sum_{k \in \sr{K}} g_{k}(\vc{Z}(t)) v_{k}(J(t)) h(J(t)) + g(\vc{Z}(t)) \sum_{j \in \sr{J}} q_{J(t),j} h(j) \nonumber\\
  & \qquad + \sum_{\ell \in \sr{K}} \left(\sum_{k \in \sr{K}} g_{k}(\vc{Z}(t)) h(J(t)) r_{k,\ell} \right) y_{\ell}(t).
\end{align}
\end{lemma}

We next choose specific functions for $g$ and $h$ for each $\vc{\theta} \in \dd{R}^{d}$. For $g$, we choose $e^{\br{\vc{\theta},\vc{z}}}$ and $h$ to be positive, that is, $h \in F_{+}(\sr{J})$. Namely,
\begin{align}
\label{eq:test 1}
  f(\vc{x}) = e^{\br{\vc{\theta},\vc{z}}} h(j), \qquad \vc{x} \equiv (\vc{z},j) \in \dd{R}_{+}^{d} \times \sr{J}.
\end{align}
Define $m \times m$ matrix $K(\vc{\theta})$ as
\begin{align*}
  K(\vc{\theta}) = \diag{\sum_{k = 1}^{d} \theta_{k} \vc{v}_{k}} + Q,
\end{align*}
where $\vc{v}_{k} = (v_{k}(1), \ldots, v_{k}(m))^{\rs{t}}$ for $k \in \sr{K}$ and $\diag{\vc{a}}$ is the $m$-dimensional diagonal matrix whose $i$-th diagonal entry is $a(i)$ for $m$-dimensional vector $\vc{a} = (a(1), \ldots, a(m))^{\rs{t}}$, then $\sr{A}f(\vc{X}(t))$ of \eq{A 1} becomes
\begin{align}
\label{eq:A 2}
 \sr{A}f(\vc{X}(t)) & = f(\vc{X}(t)) \left(\frac {[K(\vc{\theta}) \vc{h}]_{J(t)}} {h(J(t))} + \sum_{\ell \in \sr{K}} \sum_{k \in \sr{K}} \theta_{k} r_{k,\ell} y_{\ell}(t) \right).
\end{align}
Note that the set of all functions of \eq{test 1} is a sufficiently rich class for $f(\vc{z},i)$ to uniquely determine the extended generator $\sr{A}$. We present a function $h \in F_{+}(\sr{J})$ by positive vector $\vc{h}$ whose $i$-th entry is $h(i)$.

We next choose the vector $\vc{h}$ to be convenient for deriving a martingale for change of measure.
For this, we consider an eigenvalue of $K(\vc{\theta})$, and find $\vc{h}$ as its associated eigenvector. Since $Q$ is irreducible and $K(\vc{\theta})$ has nonnegative off diagonal entries, which is called an $ML$ or essentially nonnegative matrix, by Theorem 2.6 of \cite{Sene1981}, $K(\vc{\theta})$ has a unique eigenvalue $\gamma(\vc{\theta})$ such that $\gamma(\vc{\theta})$ is real and greater than the real parts of all the other eigenvalues of $K(\vc{\theta})$, that is, the Perron-Frobenius eigenvalue of $K(\vc{\theta})$. Denote the right eigenvector of $K(\vc{\theta})$ associated with $\gamma(\vc{\theta})$ by $\vc{h}_{\vc{\theta}}$. Thus, we have
\begin{align}
\label{eq:K 1}
  K(\vc{\theta}) \vc{h}_{\vc{\theta}} = \gamma(\vc{\theta}) \vc{h}_{\vc{\theta}}.
\end{align}
Since $\vc{h}_{\vc{\theta}}$ is positive and unique up to a constant multiplier, we normalize it by the left eigenvector $\vc{\xi}_{\vc{\theta}}$ of  $K(\vc{\theta})$ associated with the eigenvalue $\gamma(\vc{\theta})$ in such a way that
\begin{align}
\label{eq:normalize h1}
  \br{\vc{\xi}_{\vc{\theta}}, \vc{h}_{\vc{\theta}}} = 1, \qquad \br{\vc{\xi}_{\vc{\theta}}, \vc{1}} = 1.
\end{align}
Note that $\vc{\xi}_{\vc{0}} = (\pi(1),\pi(2),\ldots,\pi(m))^{\rs{t}}$ and $\vc{h}_{\vc{0}} = \vc{1}$ for the stationary distribution $\pi \equiv \{\pi(i); i \in \sr{J}\}$ of $J(\cdot)$ since $K(\vc{0}) = Q$. We define functions $\gamma_{k}(\vc{\theta})$ as
\begin{align*}
  \gamma_{k}(\vc{\theta}) = [\vc{\theta}^{\rs{t}} R]_{k}, \qquad k \in \sr{K}.
\end{align*}
Letting $\vc{h} = \vc{h}_{\vc{\theta}}$ and $f(\vc{z},i) = f_{\vc{\theta}}(\vc{z},i) \equiv e^{\br{\vc{\theta},\vc{z}}} h_{\vc{\theta}}(i)$ in \eq{A 2}, we have 
\begin{align}
\label{eq:A 3}
 \sr{A}f_{\vc{\theta}}(\vc{X}(t)) = f_{\vc{\theta}}(\vc{X}(t)) \left(\gamma(\vc{\theta}) + \sum_{\ell \in \sr{K}} \gamma_{\ell}(\vc{\theta}) y_{\ell}(t) \right).
\end{align}
\lem{Dynkin 1} together with this formula is a key tool for us.

\subsection{Geometric objects}
\label{sec:geometric}

As we will see in \lem{marginal 2}, if $\{\vc{X}(t); t \ge 0\}$ is a stationary process, \eq{Dynkin 1} and \eq{A 3} yield the stationary equation for appropriately chosen $\vc{\theta}$. From such an equation, it is not hard to guess that $\gamma(\vc{\theta})$ and $\gamma_{k}(\vc{\theta})$ are key characteristics for obtaining the tail decay rates. In this section, we consider geometric properties generated by these functions.

We first note that $\gamma_{k}(\vc{\theta})$ is a linear function of $\vc{\theta}$, while $\gamma(\vc{\theta})$ has the following nice analytic properties.

\begin{lemma} {\rm
\label{lem:gamma 1}
\begin{itemize}
\item [(a)] $\gamma(\vc{\theta})$ is a convex function of $\vc{\theta} \in \dd{R}^{d}$.
\item [(b)] $\gamma(\vc{\theta})$ and $\vc{h}_{\vc{\theta}}$ are continuously and partially differentiable for all $\vc{\theta} \in \dd{R}^{d}$, where  a vector valued function is differentiated in entry-wise.
\item [(c)] Let $\nabla \gamma(\vc{\theta}) = \left(\frac {\partial} {\partial \theta_{1}} {\gamma}(\vc{\theta}), \frac {\partial} {\partial \theta_{2}} {\gamma}(\vc{\theta}), \ldots, \frac {\partial} {\partial \theta_{d}} {\gamma}(\vc{\theta})\right)^{\rs{t}}$, that is, the gradient of $\gamma(\vc{\theta})$, then
\begin{align}
\label{eq:mean v 1}
  \ol{\vc{v}} = \nabla \gamma(\vc{\theta})|_{\vc{\theta} = \vc{0}},
\end{align}
\end{itemize}
}\end{lemma}

This lemma is proved in \sectn{gamma 1}. We next introduce geometric objects similar to those used in \cite{Miya2011}. For the Perron-Frobenius eigenvalue $\gamma(\vc{\theta})$, define the following sets.
\begin{align*}
& \Gamma^{-} = \{\vc{\theta} \in \dd{R}^{d}; {\gamma}(\vc{\theta}) < 0\}, \qquad \Gamma^{+} = \{\vc{\theta} \in \dd{R}^{d}; \gamma(\vc{\theta}) > 0\}.
\end{align*}
For $A \subset \sr{K}$, define the sets below, which concerns reflecting matrix $R$ through $\gamma_{k}(\vc{\theta})$.
\begin{align*}
 & \Gamma^{-}_{A} = \left\{\vc{\theta} \in \Gamma^{-}; \gamma_{k}(\vc{\theta}) < 0, \forall k \in A \right\},\quad
  \Gamma^{+}_{A} = \left\{\vc{\theta} \in \Gamma^{+}; \gamma_{k}(\vc{\theta}) \ge 0, \forall k \in A \right\}.
\end{align*}
For simplicity, $\Gamma^{-}_{\{k\}}$ and $\Gamma^{+}_{\{k\}}$ are written as $\Gamma^{-}_{k}$ and $\Gamma^{+}_{k}$, respectively, for $k \in \sr{K}$. Note that $\Gamma^{-}_{\emptyset} = \Gamma^{-}$. We define the set operations $\overleftarrow{}, \overrightarrow{}$ over set $C \subset R^{d}$ as
\begin{align*}
& \overleftarrow{C} = \{\vc{\theta} \in \dd{R}^{d}; \vc{\theta} < \vc{\theta}', \exists \vc{\theta}' \in C \}, \quad \overrightarrow{C} = \{\vc{\theta} \in \dd{R}^{d}; \vc{\theta}' < \vc{\theta}, \exists \vc{\theta}' \not\in C \}.
\end{align*}
We denote the boundary of $C$ by $\partial C $. That is, $\partial C = \ol{C} \setminus C^{\rm o}$, where $\ol{C}$ and $C^{\rm o}$ are the closure and interior of $C$, respectively. We will see that the operation $\overleftarrow{}$ and the sets $\overleftarrow{\Gamma}^{-}_{k}$'s play key roles to get the decay rates of the stationary distribution of $\vc{Z}(t)$. We here note the following lemma, which is proved in \sectn{convex 1}.

\begin{lemma} {\rm
\label{lem:convex 1}
(a)  For $A \subset \sr{K}$, $\Gamma^{-}_{A}$ and $\overleftarrow{\Gamma}^{-}_{A}$ are convex sets.

\noindent (b) Let $H_{\ol{\vc{v}}} = \{\vc{\theta} \in \dd{R}^{d}; \br{\ol{\vc{v}}, \vc{\theta}} = 0\}$, which is a hyperplane perpendicular to $\ol{\vc{v}}$. Then, $H_{\ol{\vc{v}}}$ contacts $\Gamma^{-}$ at the origin, and supports it, that is, $\Gamma^{-} \subset \{\vc{\theta} \in \dd{R}^{d}; \br{\ol{\vc{v}}, \vc{\theta}} \le 0\}$.

\noindent (c) If the stability condition \eq{stability 1} holds, then there is at least one $k$ such that the nonnegative half line of $k$-th axis, that is, $\{t \vcn{e}_{k} \in \dd{R}^{d}; t \ge 0\}$, intersects $\Gamma^{-}$ other than the origin.

\noindent (d) For each $\ell \in \sr{K}$, denote the hyperplane, $\{\vc{\theta} \in \dd{R}^{d}; \gamma_{\ell}(\vc{\theta}) = 0\}$, by $H_{\ell}$, and define ray $\sr{R}_{k} \equiv \cap_{\ell \in \sr{K} \setminus \{k\}} \{\vc{\theta} \in \dd{R}_{+}^{d}; \vc{\theta} \in H_{\ell}\}$, then the ray $\sr{R}_{k}$ intersects $\Gamma^{-}$ at other than the origin if and only if $\alpha_{k} < \mu_{k}$.
}\end{lemma}
\begin{remark}
\label{rem:convex 1}
As we will see in \rem{lower 2}, the convex set, $\Gamma^{-}$, can be unbounded in some extreme cases. 
\end{remark}

\section{Main results} \setnewcounter
\label{sec:main}

We now present main results, which are proved in Sections \sect{upper} and \sect{lower} using auxiliary lemmas, proved in \sectn{proof}. Throughout this section, we assume that the MMFN model satisfies the stability condition \eq{stability 1}, and denote the stationary distribution of $\vc{X}(t)$ by $\nu$. We derive the main results in two ways. We first use the moment generating function of $\vc{Z}(t)$ under $\nu$, and derive bounds for the tail decay rates from its finite domain. This method is only applicable when the tail set is an envelope of hyperplanes. We secondly use change of measure to present the tail probability under a new measure generated by an exponential martingale. This method is typically used in the study of large deviations, but our problem may be harder because of reflecting mechanism.

Before presenting those two methods, we introduce stationary framework and Palm distributions.

\subsection{Stationary framework}
\label{sec:stationary}

Assume that $\vc{X}(0)$ has the stationary distribution $\nu$, then $\{\vc{X}(t); t \ge 0\}$ is a stationary process and can be extended to $\{\vc{X}(t); t \in \dd{R}\}$ on the whole line. For this stationary process, we introduce a shift operator group $(\Theta_{s} \circ \Theta_{t})(\omega) = \Theta_{s+t}(\omega)$ on the sample space $\Omega$ which operates on $\{\vc{X}(t); t \in \dd{R}\}$ as
\begin{align*}
  (\vc{X}(s) \circ \Theta_{t}) (\omega) = \vc{X}(s+t)(\omega).
\end{align*}
Denote the probability measure for this stationary process by $\dd{P}_{\nu}$, and simply denote $(\vc{Z}(0), J(0))$ under it by $(\vc{Z}, J)$. As usual, $\dd{E}_{\nu}$ denotes expectation concerning $\dd{P}_{\nu}$.

Since $Y_{k}(t)$ has stationary increments under $\dd{P}_{\nu}$, we may view it as a random measure whose increments are jointly stationary with $\vc{X}(\cdot)$. Note the following fact, which is proved in \sectn{Palm 1}.
\begin{lemma}
\label{lem:Palm 1}
(a) $\dd{E}_{\nu}(Y_{k}(1)) = - \dd{E}_{\nu}([R^{-1} \vc{V}(1)]_{k}) = \ol{\mu}_{k} - \ol{\alpha}_{k} > 0$ for all $k \in \sr{K}$.\\
(b) $\dd{P}_{\nu}(\vc{Z}(0) \in S_{k}) > 0$ for all $k \in \sr{K}$.
\end{lemma}
Let $m_{k} = \dd{E}_{\nu}(Y_{k}(1))$, then we can define a probability measure on $(\Omega, \sr{F})$ as
\begin{align}
\label{eq:Palm 1}
  \dd{P}_{k}(D) = m_{k}^{-1} \dd{E}_{\nu}\left(\int_{0}^{1} 1_{D} \circ \Theta_{u} Y_{k}(du)\right), \qquad D \in \sr{F}, k \in \sr{K},
\end{align}
which is referred to as a Palm distribution with respect to $Y_{k}(\cdot)$ (e.g., see Chapter 11 of \cite{Kall2001}). Roughly speaking, Palm distribution $\dd{P}_{k}$ presents the probability of events observed at time $t$ which is weighted by the increments of $Y_{k}(t)$, so may be considered as a certain conditional probability measure. Palm distributions are quite useful to derive relationship among various characteristics of queueing models (e.g., see \cite{Miya1994}).

In the next section we will consider the set of moment generating functions. For variable $\vc{\theta} \in \dd{R}^{d}$, define moment generating functions as
\begin{align}
\label{eq:set 1}
 & \varphi(\vc{\theta}) = \dd{E}_{\nu}(e^{\br{\vc{\theta},\vc{Z}}}),\qquad
  \varphi_{k}(\vc{\theta}) = m_{k} \dd{E}_{k}(e^{\br{\vc{\theta},\vc{Z}}}) \quad k \in \sr{K}.
\end{align}
It should be noted that $\varphi_{k}(\vc{\theta})$'s are weighted by $m_{k}$. This make some expressions simpler. For example, $\varphi_{k}(\vc{\theta})$ is simply expressed in term of $\dd{P}_{\nu}$ by
\begin{align*}
  \varphi_{k}(\vc{\theta}) = \dd{E}_{\nu}\left(\int_{0}^{1} e^{\br{\vc{\theta},\vc{Z}(u)}} Y_{k}(du) \right).
\end{align*}
Intuitively, $\varphi_{k}(\vc{\theta})$'s present the influence of reflection at the boundaries. They play important roles to find the domain of $\varphi(\vc{\theta})$ as we will see below.

\subsection{Asymptotics through moment generating functions}
\label{sec:asymptotic 1}

In this subsection, we presents results on the tail asymptotics of the stationary distribution obtained from moment generating functions. Define the domains of $\varphi(\vc{\theta})$ and $\varphi_{k}(\vc{\theta})$ as
\begin{align*}
  \sr{D} = \{\vc{\theta} \in \dd{R}^{d}; \varphi(\vc{\theta}) < \infty\}, \qquad \sr{D}_{k}(\vc{\theta}) = \{\vc{\theta} \in \dd{R}^{d}; \varphi_{k}(\vc{\theta}) < \infty\}, \quad k \in \sr{K}.
\end{align*}

Let $\sr{U}^{d} = \{\vc{x} \in \dd{R}^{d}; \|\vc{x}\| = 1\}$, and let $\sr{U}^{d}_{+} = \{\vc{x} \in \sr{U}^{d}; \vc{x} \ge \vc{0}\}$, which is the set of direction vectors for the tail asymptotics. Obviously, the domain $\sr{D}$ is useful to get the tail decay rates. In particular, for $\vc{c} \in \sr{U}^{d}_{+}$ and $\alpha \ge 0$,
\begin{align}
\label{eq:mgf 1}
  e^{\alpha x} \dd{P}_{\nu}(\br{\vc{c},\vc{Z}} > x) \le \dd{E}_{\nu}(e^{\alpha\br{\vc{c},\vc{Z}}}) = \varphi(\alpha \vc{c}) < \infty \quad \mbox{if} \quad \alpha \vc{c} \in \sr{D}.
\end{align}
This suggests the following lemma, which is proved in \sectn{lem-upper 1}.
\begin{lemma} %\rm
\label{lem:upper 1}
Assume that the MMFN is stable. (a) For $\vc{c} \in \sr{U}^{d}_{+}$,
\begin{align}
\label{eq:mgf 2}
  \limsup_{x \to \infty} \frac 1x \log \dd{P}_{\nu}(\br{\vc{c},\vc{Z}} > x) = - \sup\{\alpha \ge 0; \alpha \vc{c} \in \sr{D}\}.
\end{align}
(b) Let $B_{0}$ be a compact subset of $\dd{R}_{+}^{d}$. For $k \in \sr{K}$ and for $\vc{c} \in \sr{U}^{d}_{+}$,
\begin{align}
\label{eq:upper bound 1}
 & \limsup_{x \to \infty} \frac 1x \log \dd{P}(\vc{Z} \in x \vc{c} + B_{0}) \le - \sup\{\br{\vc{\theta}, \vc{c}}; \vc{\theta} \in \sr{D}\}.
\end{align}
\end{lemma}

To make these results useful, we need to find the domain $\sr{D}$. This is a big problem. To solve this problem, we consider the following fixed point equation for $(d-1)$-dimensional sets $D_{k}$, motivated by the cases for $d=2,3$, which will be discussed in \sectn{motivation}.
\begin{align}
\label{eq:fixed 1}
  D_{k} = \bigcup_{k \in A \subset \sr{K}} \{\vc{\theta}_{\sr{K} \setminus \{k\}} \in \dd{R}^{d-1}; \vc{\theta} \in \overleftarrow{\Gamma}^{-}_{A}, \vc{\theta}_{\sr{K} \setminus \{\ell\}} \in \overleftarrow{D_{\ell}}, \forall \ell \in \sr{K} \setminus A\}, \quad k \in \sr{K}.
\end{align}
Once these $D_{k}$'s are obtained, $\{\vc{\theta} \in \overleftarrow{\Gamma}^{-}; \vc{\theta}_{\sr{K} \setminus \{k\}} \in D_{k}, \forall k \in \sr{K}\}$ is expected to be the domain $\sr{D}$. To find the solution of \eq{fixed 1}, we will apply the following iteration. Let $\overleftarrow{\vc{0}}_{\sr{K} \setminus \{k\}} = \{\vc{\theta}_{\sr{K} \setminus \{k\}} \in \dd{R}^{d-1}; \vc{\theta} < \vc{0}\}$, and let $D^{(0)}_{k} = \overleftarrow{\vc{0}}_{\sr{K} \setminus \{k\}}$ for $k \in \sr{K}$, and inductively define $D^{(n)}_{k}$ for $n \ge 1$ as
\begin{align}
\label{eq:fixed 2}
  D^{(n)}_{k} = \bigcup_{k \in A \subset \sr{K}} \{\vc{\theta}_{\sr{K} \setminus \{k\}} \in \dd{R}^{d-1}; \vc{\theta} \in \overleftarrow{\Gamma}^{-}_{A}, \vc{\theta}_{\sr{K} \setminus \{\ell\}} \in \overleftarrow{D}^{(n-1)}_{\ell}, \forall \ell \in \sr{K} \setminus A\}, \; k \in \sr{K}.
\end{align}
Obviously, $D^{(0)}_{k} \subset D^{(1)}_{k}$ for $\forall k \in \sr{K}$, then the limit of the iterations gives the solution of \eq{fixed 1}. However, for this solution to be meaningful, it must be nontrivial, that is, $D^{(1)}_{k} \ne \overleftarrow{\vc{0}}_{\sr{K} \setminus \{k\}}$ for $\exists k \in \sr{K}$. This should be checked. Including this fact, we present results about the fixed point equation \eq{fixed 1} below. Let
\begin{align*}
  \widetilde{\sr{D}}_{k} = \{\vc{\theta}_{\sr{K} \setminus \{k\}} \in \dd{R}^{d-1}; \varphi_{k}(\vc{\theta}) < \infty\}, \qquad k \in \sr{K}.
\end{align*}

\begin{lemma} \rm
\label{lem:fixed 1}
Assume that the MMFN process is stable. (a) For $D^{(0)}_{k} = \overleftarrow{\vc{0}}_{\sr{K} \setminus \{k\}}$ for $k \in \sr{K}$, $\{D^{(1)}_{k}; k \in \sr{K}\}$ obtained from the iteration \eq{fixed 2} is nontrivial.\\
(b) There is a nontrivial solution $\{D^{(\max)}_{k}; k \in \sr{K}\}$ of \eq{fixed 1} all of whose elements are maximal in the order of set inclusion, and
\begin{align}
\label{eq:Dd 2}
 & D^{(\max)}_{k} \subset \widetilde{\sr{D}}_{k}, \qquad k \in \sr{K},\\
\label{eq:Dd 3}
 & D^{(\max)} \equiv \{\vc{\theta} \in \overleftarrow{\Gamma}^{-}; \vc{\theta}_{\sr{K} \setminus \{k\}} \in D^{(\max)}_{k}, \forall k \in \sr{K}\} \subset \sr{D}.
\end{align}
(c) For $d=2$, $\{D^{(\max)}_{k}; k \in \sr{K}\}$ is a unique solution of \eq{fixed 1} such that $D^{(\max)}_{k} = \widetilde{\sr{D}}_{k}$ for $k=1,2$ and $D^{(\max)} = \sr{D}$.
\end{lemma}

This lemma is proved in \sectn{fixed 1}. We conjecture that $D^{(\max)} = \sr{D}$ for $d \ge 3$, but have not yet proved it. So, we content to use $D^{(\max)}$ in place for $\sr{D}$ for $d \ge 3$ in this paper.

\begin{theorem} %\rm
\label{thr:upper 1}
Assume that the MMFN process satisfies (\sect{MMFN}a), (\sect{MMFN}b) and the stability condition \eq{stability 1}. Let $D^{(\max)}$ be the set obtained in \lem{fixed 1}. (a) For $\vc{c} \in \sr{U}^{d}_{+}$,
\begin{align}
\label{eq:upper 1}
  \limsup_{x \to \infty} \frac 1x \log \dd{P}_{\nu}(\br{\vc{c},\vc{Z}} > x) \le - \sup\{\alpha \ge 0; \alpha \vc{c} \in D^{(\max)}\}.
\end{align}
(b) Let $B_{0}$ be a compact subset of $\dd{R}_{+}^{d}$. For $k \in \sr{K}$ and for $\vc{c} \in \sr{U}^{d}_{+}$,
\begin{align}
\label{eq:upper 2}
 & \limsup_{x \to \infty} \frac 1x \log \dd{P}(\vc{Z} \in x \vc{c} + B_{0}) \le - \sup\{\br{\vc{\theta}, \vc{c}}; \vc{\theta} \in D^{(\max)}\}.
\end{align}
(c) In particular, for $d=2$, $\sr{D} = D^{(\max)}$ and
\begin{align}
\label{eq:upper 3}
 & \lim_{x \to \infty} \frac 1x \log \dd{P}(\vc{Z} \in x \vc{e}_{k} + B_{0}) = - \sup\{\theta_{k}; \vc{\theta} \in D^{(\max)}\}, \qquad k=1,2,\\
\label{eq:upper 4}
 & \lim_{x \to \infty} \frac 1x \log \dd{P}_{\nu}(\br{\vc{c},\vc{Z}} > x) = - \sup\{\alpha \ge 0; \alpha \vc{c} \in D^{(\max)}\}.
\end{align}
\end{theorem}
This theorem is proved in \sectn{upper 1}.

\subsection{Asymptotics through change of measure}
\label{sec:asymptotic 2}

We next presents results on the tail asymptotics obtained from change of measure. In principle, this change of measure approach can be used for any tail set and both of upper and lower bounds. However, the upper bounds which can be obtained from the change of measure formula \eq{exp 1} are essentially the same as those obtained in \thr{upper 1} as discussed in the end of \sectn{tail} (see \lem{W 1}). So, we here focus on lower bounds. For $B \subset \dd{R}_{+}^{d}$, define a convex corn as
\begin{align*}
  {\rm Corn}(B) = \{\vc{x} \in \dd{R}_{+}^{d}; u \vc{x} \in B, \exists u > 0\},
\end{align*}
and, for $k \in \sr{K}$, define
\begin{align*}
  G_{k} = \Big\{\vc{\theta} \in \Gamma^{+}_{\sr{K} \setminus \{k\}}; \frac {\partial} {\partial \theta_{\ell}} \gamma(\vc{\theta}) < 0, \ell \in \sr{K} \setminus \{k\}, [R^{-1} \nabla \gamma(\vc{\theta})]_{k} > 0\Big\}.
\end{align*}

\begin{theorem} %\rm
\label{thr:lower 1}
Under the same assumptions as in \thr{upper 1},\\
(a) For $k \in \sr{K}$ and a compact set $B_{0} \in \dd{R}_{+}^{d}$,
\begin{align}
\label{eq:lower 1}
 & \liminf_{x \to \infty} \frac 1x \log \dd{P}(\vc{Z} \in x \vcn{e}_{k} + B_{0}) \ge - \inf\{\theta_{k} \ge 0; \vc{\theta} \in G_{k}\},
\end{align}
(b) For $\vc{c} \in \sr{U}^{d}_{+}$ if $\vc{c} \in {\rm Corn}(\overleftarrow{\Gamma}^{-} \cap \partial \Gamma^{-})$,
\begin{align}
\label{eq:lower 2}
 & \liminf_{x \to \infty} \frac 1x \log \dd{P}(\br{\vc{c},\vc{Z}} > x) \ge - \inf\big\{u \ge 0; u \vc{c} \in \overrightarrow{\Gamma}^{+} \big\}.
\end{align}
\end{theorem}

This theorem is proved in \sectn{lower 1}.

\begin{remark}
\label{rem:lower 2}
(a) By \lem{upper 1}, \eq{lower 2} implies that $\sr{D} \subset \overleftarrow{\Gamma}^{-}$.\\
(b) The decay rate in the $k$-th coordinate direction, $- \lim_{x \to \infty} \frac 1x \log \dd{P}(Z_{k} > x)$, is infinite when the buffer of station $k$ is always empty under the stationary distribution. This occurs if $v_{k}(i) < 0$ for all $i \in \sr{J}$. 
In this case, by (a),  the convex set $\overleftarrow{\Gamma}^{-}$ is unbounded in the direction of $k$-th coordinate direction.
\end{remark}
\begin{remark}\rm
\label{rem:lower 1}
For $d = 2$, \eq{lower 2} is identical with the exact decay rate of \eq{upper 4} for $\vc{c} \in {\rm Corn}(\overleftarrow{\Gamma}^{-} \cap \partial \Gamma^{-})$. This is also the case for \eq{lower 1} and \eq{upper 3} for $k=1$ if $\theta_{2} \le \alpha_{2}$ for $\vc{\theta} \in G_{1}$ in \eq{lower 1}, where $\vc{\alpha}$ is the solution of \eq{cp 1a} in \sectn{motivation} below.
\end{remark}

For $d \ge 3$, there is very little known about the tail asymptotics of the stationary distribution even for simple reflecting processes like a reflecting random walk and SRBM.  There are several studies in the framework of sample path large deviations (e.g., see \cite{DupuElli1997,FreiWent1998}). In particular, local rate functions are explicitly obtained for some simple networks such as the Jackson network (e.g., see \cite{Igna2000,Igna2005}). Those results are applicable for various asymptotic problems, but, for the tail decay rates, local rate functions are not sufficient. We need to compute large deviations rate functions from them. This is the problem to find an optimal path in a multidimensional orthant, which is not easy, and requires considerable effort even for the two dimensional case (see \cite{AvraDaiHase2001}). There are some numerical studies, but still complicated (e.g., see \cite{Maje2009}). Compared with them, the bounds in Theorems \thrt{upper 1} and \thrt{lower 1} may be more tractable.

\section{Stationary formulas and the proof of \thr{upper 1}} \setnewcounter
\label{sec:upper}

A main subject of this section is to prove \lem{fixed 1}, from which \thr{upper 1} is proved. 

\subsection{Stationary inequality}
\label{sec:inequality}

\lem{fixed 1} is essentially results for moment generating functions $\varphi(\vc{\theta})$ and $\varphi_{k}(\vc{\theta})$. To consider them, we derive a stationary equation for $\vc{X}(t)$ using an appropriate test function. For this, we use Dynkin's formula \eq{Dynkin 1} with the test function \eq{test 1}. However, this formula are not for $\varphi(\vc{\theta})$ and $\varphi_{k}(\vc{\theta})$ because it includes the background process $J(\cdot)$. So, we consider another set of moment generating functions.
\begin{align}
\label{eq:set 2}
 & \psi(\vc{\theta}) = \dd{E}_{\nu}(e^{\br{\vc{\theta},\vc{Z}}} h(J)),\qquad
  \psi_{k}(\vc{\theta}) = m_{k} \dd{E}_{k}(e^{\br{\vc{\theta},\vc{Z}}} h(J)) \quad k \in \sr{K}.
\end{align}
This is inconvenient for proving \lem{fixed 1}, but unavoidable. However, fortunately, this second set of moment generating functions is equivalent to the set of $\varphi(\vc{\theta})$ and $\varphi_{k}(\vc{\theta})$, which is referred to as the first set, concerning their finiteness as summarized in the following lemma.
\begin{lemma}\rm
\label{lem:marginal 1}
(a) For $\vc{\theta} \in \dd{R}^{d}$, $\varphi(\vc{\theta}) < \infty$ if and only if $\psi(\vc{\theta}) < \infty$.\\
(b) For $\vc{\theta} \in \dd{R}^{d}$ and $k \in \sr{K}$, the following three finiteness conditions are equivalent.
\begin{align*}
  \varphi_{k}(\vc{\theta}) < \infty, \qquad \psi_{k}(\vc{\theta}) < \infty, \qquad \dd{E}_{\nu}\left( e^{\br{\vc{\theta},\vc{Z}}} 1(\vc{Z} \in S_{k}) \right) < \infty.
\end{align*}
\end{lemma}
\begin{proof}
Both (a) and (b) are immediate from \lem{Palm 1} since $0 \le y_{k}(t) \le \max_{i \in \sr{J}} \mu_{k}(i)$ and $0 < \min_{j \in \sr{J}} h(j) \le \max_{j \in \sr{J}} h(i) < \infty$.
\end{proof} 

Thus, we work on $\psi(\vc{\theta})$ and $\psi_{k}(\vc{\theta})$ instead of the first set. To get their relationship, we apply a similar approach to the one used for a two-dimensional semimartingale reflecting Brownian motion (SRBM for short) in \cite{DaiMiya2011} (see also Lemma 6.1 of \cite{Miya2011} for a reflecting random walk). In those papers, the first set of generating functions \eq{set 1} is used, but we can work on the second set due to \lem{marginal 1}. The following lemma is a key for this approach, which is proved in \sectn{marginal 2}. 

\begin{lemma}\rm
\label{lem:marginal 2}
(a) For $\vc{\theta} \in \dd{R}^{d}$, $\varphi(\vc{\theta}) < \infty$, equivalently, $\psi(\vc{\theta}) < \infty$, then $\varphi_{k}(\vc{\theta}) < \infty$ and $\psi_{k}(\vc{\theta}) < \infty$ for $\forall k \in \sr{K}$, and
\begin{align}
\label{eq:BAR 1}
  \gamma(\vc{\theta}) \psi(\vc{\theta}) + \sum_{k =1}^{d} \gamma_{k}(\vc{\theta}) \psi_{k}(\vc{\theta}) = 0.
\end{align}
(b) For $A \subset \sr{K}$ and $\vc{\theta} \in \Gamma^{-}_{\sr{K} \setminus A}$, if $\varphi_{k}(\vc{\theta}) < \infty$, equivalently, $\psi_{k}(\vc{\theta}) < \infty$ for $\forall k \in A$, then
\begin{align}
\label{eq:BAR 2}
  (-\gamma(\vc{\theta})) \psi(\vc{\theta}) + \sum_{k \in \sr{K} \setminus A} (-\gamma_{k}(\vc{\theta})) \psi_{k}(\vc{\theta}) \le \sum_{k \in A} \gamma_{k}(\vc{\theta}) \psi_{k}(\vc{\theta}),
\end{align}
and therefore $\varphi(\vc{\theta}) < \infty$, $\varphi_{k}(\vc{\theta}) < \infty$ for $k \in \sr{K} \setminus A$, equivalently, $\psi(\vc{\theta}) < \infty$, $\psi_{k}(\vc{\theta}) < \infty$ for $k \in \sr{K} \setminus A$, and \eq{BAR 1} holds.
\end{lemma}

\subsection{The motivation for the fixed point equation}
\label{sec:motivation}

In what follows, we discuss how we arrived at the fixed point equation \eq{fixed 1} for finding the domain $\sr{D}$. We first consider this for $d=2$. From \lem{marginal 2}, we can derive the domain by applying similar arguments in \cite{DaiMiya2011} (see Theorem 2.1 there).

\begin{lemma}
\label{lem:domain 2d}
For the two-dimensional MMFN, if it is stable, then the fixed point equation:
\begin{align}
\label{eq:cp 1a}
 & \alpha_{1} = \sup\{ \theta_{1} \ge 0; \vc{\theta} \in \Gamma^{-}_{2}, \theta_{2} \le \alpha_{2} \}, \quad \alpha_{2} = \sup\{ \theta_{2} \ge 0; \vc{\theta} \in \Gamma^{-}_{1}, \theta_{1} \le \alpha_{1} \}.
\end{align}
has a unique solution $\vc{\alpha} = (\alpha_{1}, \alpha_{2})$, and
\begin{align}
\label{eq:cp 1c}
   \sr{D} = \{\vc{\theta} \in \overleftarrow{\Gamma}^{-}; \theta_{i} < \alpha_{i}, i=1,2\}.
\end{align}
\end{lemma}
\begin{remark}\rm
\label{rem:domain 2d}
A concrete shape of $\sr{D}$ is obtained in \cite{DaiMiya2011,Miya2011}.
\end{remark}

We like to generalize the fixed point equation \eq{cp 1a} for $d \ge 3$. The first idea is to replace $\alpha_{1}$ and $\alpha_{2}$ by vectors, but the supremes as in \eq{cp 1a} can not be used. To overcome this difficulty, we replace $\alpha_{1}$ and $\alpha_{2}$ by sets, $D_{2} = \{\theta_{1} \in \dd{R}; \theta_{2} \le \alpha_{2}\}$ and $D_{1} = \{\theta_{2} \in \dd{R}; \theta_{1} \le \alpha_{1}\}$, then \eq{cp 1a} and \eq{cp 1c} are rewritten as
\begin{align}
\label{eq:cp 2a}
 & D_{2} = \{ \theta_{1} \in \dd{R}; \vc{\theta} \in \overleftarrow{\Gamma}^{-}_{2}, \theta_{2} \in \overleftarrow{D}_{1} \}, \quad D_{1} = \{ \theta_{2} \in \dd{R}; \vc{\theta} \in \overleftarrow{\Gamma}^{-}_{1}, \theta_{1} \in \overleftarrow{D}_{2} \},\\
\label{eq:cp 2c}
  & \sr{D} = D \equiv \{\vc{\theta} \in \overleftarrow{\Gamma}^{-}; \theta_{1} \in D_{2}, \theta_{2} \in D_{1}\},
\end{align}
where \eq{cp 2a} is the fixed point equation for one-dimensional sets $D_{k}$ for $k=1,2$. By \lem{domain 2d}, this proves \lem{fixed 1} for $d=2$.

We apply this idea of the fixed point equation of sets for $d \ge 3$. For example, for $d=3$, one may consider to replace the two set equations in \eq{cp 2a} by the following three set equations.
\begin{align*}
 & D_{1} = \{(\theta_{2}, \theta_{3}) \in \dd{R}^{2}; \vc{\theta} \in \overleftarrow{\Gamma}^{-}_{1}, (\theta_{1},\theta_{2}) \in \overleftarrow{D}_{3}, (\theta_{1},\theta_{3}) \in D_{2}\},\\
 & D_{2} = \{(\theta_{1}, \theta_{3}) \in \dd{R}^{2}; \vc{\theta} \in \overleftarrow{\Gamma}^{-}_{2}, (\theta_{2},\theta_{3}) \in \overleftarrow{D}_{1}, (\theta_{1},\theta_{2}) \in D_{3}\},\\
 & D_{3} = \{(\theta_{1}, \theta_{2}) \in \dd{R}^{2}; \vc{\theta} \in \overleftarrow{\Gamma}^{-}_{3}, (\theta_{1},\theta_{3}) \in \overleftarrow{D}_{2}, (\theta_{2},\theta_{3}) \in D_{1}\}.
\end{align*}
However, this is not good because $D^{(0)}_{k} = \overleftarrow{\vc{0}}_{\sr{K} \setminus \{k\}}$ for $k=1,2,3$ imply $D^{(1)}_{k} = \overleftarrow{\vc{0}}_{\sr{K} \setminus \{k\}}$. Hence, we must add more regions in their righthand sides fully utilizing \eq{BAR 2}. For example, the above equation for $D_{1}$ should be replaced by
\begin{align*}
  D_{1} = \bigcup_{1 \in A \subset \sr{K}} \{(\theta_{2},\theta_{3}) \in \dd{R}^{d-1}; \vc{\theta} \in \overleftarrow{\Gamma}^{-}_{A}, \vc{\theta}_{\sr{K} \setminus \{\ell\}} \in \overleftarrow{D_{\ell}}, \forall \ell \in \sr{K} \setminus A\}.
\end{align*}
 This is the reason why we arrive at the fixed point equation \eq{fixed 1}. 

\subsection{Proof of \thr{upper 1}}
\label{sec:upper 1}

Lemmas \lemt{upper 1} and \lemt{fixed 1} yield (a) and (b), while (c) is obtained from Lemmas \lemt{marginal 1} and \lemt{marginal 2}, using the same arguments in \cite{DaiMiya2011,DaiMiya2013}. See also (c) of \lem{fixed 1} and \lem{domain 2d}.

\section{Change of measure and the proof of \thr{lower 1}} \setnewcounter
\label{sec:lower}

In this section, we aim to derive the stationary tail probabilities of $\vc{Z}(t)$ for asymptotic analysis through exponential change of measure. After presenting several lemmas for this, we prove \thr{lower 1}. 

\subsection{MMFN process under change of measure} 
\label{sec:martingale}

We consider the $d$-dimensional MMFN process under change of measure by an exponential martingale. Recall that $\vc{h}_{\vc{\theta}}$ is the normalized right eigenvector of $K(\vc{\theta})$ for the eigenvalue $\gamma(\vc{\theta})$ and $f_{\vc{\theta}}(\vc{z},i) \equiv e^{\br{\vc{\theta},\vc{z}}} h_{\vc{\theta}}(i)$. It follows from \eq{A 3} and \lem{Dynkin 1} that
\begin{align}
\label{eq:M 2}
  M(t) & = f_{\vc{\theta}}(\vc{X}(t)) - f_{\vc{\theta}}(\vc{X}(0)) - \int_{0}^{t} f_{\vc{\theta}}(\vc{X}(s)) \left(\gamma(\vc{\theta}) + \sum_{k \in \sr{K}} \gamma_{k}(\vc{\theta}) y_{\ell}(s) \right) ds
\end{align}
is an $\dd{F}$-martingale.

By the standard arguments (e.g., see Section 3.1 of \cite{Miya2017a}), we can define an exponential martingale $E^{\vc{\theta}}(t)$ as
\begin{align}
\label{eq:EM 1}
  E^{\vc{\theta}}(t) & = \frac {e^{\br{\vc{\theta},\vc{Z}(t)}} h_{\vc{\theta}}(J(t))} {e^{\br{\vc{\theta},\vc{Z}(0)}} h_{\vc{\theta}}(J(0))} \exp\left(- \gamma(\vc{\theta}) t - \sum_{k \in \sr{K}} \gamma_{k}(\vc{\theta}) Y_{k}(t)\right), \qquad t \ge 0.
\end{align}

We then change the probability measure $\dd{P}_{\vc{x}}$ by $E^{\vc{\theta}}(\cdot)$ (e.g., see \cite{KuniWata1963,Miya2017a}). Namely, we define a probability measure $\widetilde{\dd{P}}_{\vc{x}}^{\vc{\theta}}$ under an initial state $\vc{x} \in S$ by
\begin{align}
\label{eq:change 1}
  \frac {d\widetilde{\dd{P}}_{\vc{x}}^{\vc{\theta}}} {d\dd{P}_{\vc{x}}} \Big|_{\sr{F}_{t}} = E^{\vc{\theta}}(t), \qquad t \ge 0.
\end{align}
Let $\dd{P}_{\nu}$ and $\widetilde{\dd{P}}_{\nu}^{\vc{\theta}}$ be probability measures such that $\dd{P}_{\nu}(C) = \int_{S} \dd{P}_{\vc{x}}(C)\nu(d\vc{x})$ and $\widetilde{\dd{P}}_{\nu}^{\vc{\theta}}(C) = \int_{S} \widetilde{\dd{P}}^{\vc{\theta}}_{\vc{x}}(C)\nu(d\vc{x})$ for $C \in \sr{F}$ and the stationary distribution $\nu$, namely, those under $\vc{X}(0)$ subject to $\nu$ on $S$, then, for a non-negative $\sr{F}_{t}$-measurable random variable $U$ with finite expectation, we have
\begin{align}
\label{eq:change 2}
   \widetilde{\dd{E}}^{\vc{\theta}}_{\nu}(U) = \dd{E}_{\nu}(UE^{\vc{\theta}}(t)), \qquad \dd{E}_{\nu}(U) = \widetilde{\dd{E}}^{\vc{\theta}}_{\nu}(E^{\vc{\theta}}(t)^{-1}U),
\end{align}
where $\dd{E}_{\nu}$ and $\widetilde{\dd{E}}^{\vc{\theta}}_{\nu}$ represent the expectations concerning $\dd{P}_{\nu}$ and $\widetilde{\dd{P}}_{\nu}^{\vc{\theta}}$, respectively. Similarly, for conditional expectations, we have, for $0 \le s < t$ and $\sr{F}_{t}$-measurable $U$,
\begin{align}
\label{eq:change 3}
  \widetilde{\dd{E}}^{\vc{\theta}}_{\nu}(U|\sr{F}_{s}) = \dd{E}_{\nu}\Big(U\frac {E^{\vc{\theta}}(t)}{E^{\vc{\theta}}(s)} \Big|\sr{F}_{s}\Big), \qquad \dd{E}_{\nu}(U |\sr{F}_{s}) = \widetilde{\dd{E}}^{\vc{\theta}}_{\nu}\Big(U\frac {E^{\vc{\theta}}(s)}{E^{\vc{\theta}}(t)} \Big|\sr{F}_{s}\Big).
\end{align}

Our next task is to find the process $\vc{X}(\cdot)$ under the new measure $\widetilde{\dd{P}}_{\nu}^{\vc{\theta}}$. It is known that $\vc{X}(\cdot)$ is a Markov process under $\widetilde{\dd{P}}_{\nu}^{\vc{\theta}}$, and its extended generator $\widetilde{\sr{A}}^{\vc{\theta}}$ is given for test function $\widetilde{f} \in C^{1}(S)$ by
\begin{align}
\label{eq:new A 1}
  \widetilde{\sr{A}} \widetilde{f}(\vc{x}) = \frac 1{f_{\vc{\theta}}(\vc{x})} \left( \sr{A}(\widetilde{f} \bullet f_{\vc{\theta}})(\vc{x}) - \widetilde{f}(\vc{x}) \sr{A} f_{\vc{\theta}}(\vc{x}) \right), \quad \vc{x} \in S \times \sr{J}, \;\; \widetilde{f} \bullet f_{\vc{\theta}} \in \sr{D}(\sr{A}),
\end{align}
where $\widetilde{f} \bullet f_{\vc{\theta}}(\vc{x}) = \widetilde{f}(\vc{x}) f_{\vc{\theta}}(\vc{x})$ and $\sr{D}(\sr{A})$ is the domain of $\sr{A}$ (see, e.g., \cite{PalmRols2002}).

For $\vc{\eta} \in \dd{R}^{d}$ and $\widetilde{h} \in F_{+}(\sr{J})$, $\widetilde{f}(\vc{z},i) = e^{\br{\vc{\eta},\vc{z}}} \widetilde{h}(i)$. By \eq{A 2} and \eq{A 3}, we have
\begin{align*}
 & \frac {\sr{A}(\widetilde{f} \bullet f_{\vc{\theta}})(\vc{X}(t))} {f_{\vc{\theta}}(\vc{X}(t))} = e^{\br{\vc{\eta},\vc{Z}(t)}} \left(\frac {[K(\vc{\eta} + \vc{\theta}) \widetilde{\vc{h}} \bullet \vc{h}_{\vc{\theta}}]_{J(t)}} {h_{\vc{\theta}}(J(t))} +\widetilde{h}(J(t)) \sum_{k \in \sr{K}} \gamma_{k}(\vc{\eta}+ \vc{\theta}) y_{k}(t) \right),\\
 & \frac 1{f_{\vc{\theta}}(\vc{X}(t))} \widetilde{f}(\vc{X}(t)) \sr{A} f_{\vc{\theta}}(\vc{x}) = e^{\br{\vc{\eta},\vc{Z}(t)}} \widetilde{h}(J(t)) \left(\gamma(\vc{\theta}) + \sum_{k \in \sr{K}} \gamma_{k}(\vc{\theta}) y_{k}(t) \right),
\end{align*}
where $\widetilde{\vc{h}} \bullet \vc{h}_{\vc{\theta}}(i) = \widetilde{h}(i) h_{\vc{\theta}}(i)$. Hence, \eq{new A 1} yields
\begin{align}
\label{eq:new A 2}
  \widetilde{\sr{A}} \widetilde{f}(\vc{X}(t)) = \widetilde{f}(\vc{X}(t)) \left(\frac {[K(\vc{\eta} + \vc{\theta}) \widetilde{\vc{h}} \bullet \vc{h}_{\vc{\theta}}]_{J(t)}} {\widetilde{h}(J(t)) h_{\vc{\theta}}(J(t))} - \gamma(\vc{\theta}) + \sum_{k \in \sr{K}} \gamma_{k}(\vc{\eta}) y_{k}(t) \right).
\end{align}

We now define $m \times m$ matrix $Q^{\vc{\theta}} = \{q_{ij}^{\vc{\theta}}; i,j \in \sr{J}\}$ as
\begin{align*}
  q_{ij}^{\vc{\theta}} = \left\{
\begin{array}{ll}
 q_{ij} h_{\vc{\theta}}(j)/h_{\vc{\theta}}(i), \quad & i \ne j,  \\
  \sum_{k \in \sr{K}} \theta_{k} v_{k}(i) + q_{ii} - \gamma(\vc{\theta}), & i=j,
\end{array}
\right.
\end{align*}
then $Q^{\vc{\theta}}$ is a transition rate matrix because it follows from \eq{K 1} that
\begin{align*}
  \sum_{j \in \sr{J}} q_{ij}^{\vc{\theta}} = \frac {1}{h_{\vc{\theta}}(i)} \left(\sum_{j \ne i} q_{ij} h_{\vc{\theta}}(j) + \left(\sum_{k \in \sr{K}} \theta_{k} v_{k}(i) + q_{ii}\right) h_{\vc{\theta}}(i) - \gamma(\vc{\theta}) h_{\vc{\theta}}(i) \right) = 0.
\end{align*}
Furthermore, we have
\begin{align}
\label{eq:new K 1}
  \frac {[K(\vc{\eta} + \vc{\theta}) \widetilde{\vc{h}} \bullet \vc{h}_{\vc{\theta}}]_{i}} {\widetilde{h}(i) h_{\vc{\theta}}(i)} - \gamma(\vc{\theta}) & = \sum_{k \in \sr{K}} (\eta_{k} + \theta_{k}) v_{k}(i) + \sum_{j \in \sr{J}} q_{ij} \frac {h_{\vc{\theta}}(j) \widetilde{h}(j)} {h_{\vc{\theta}}(i) \widetilde{h}(i)}  - \gamma(\vc{\theta}) \nonumber\\
  & = \frac{1}{\widetilde{h}(i)} \left(\sum_{k \in \sr{K}} \eta_{k} v_{k}(i) \widetilde{h}(i) + \sum_{j \in \sr{J}} q_{ij}^{\vc{\theta}} \widetilde{h}(j)\right).
\end{align}

Thus, defining $K^{\vc{\theta}}(\vc{\eta}) \equiv \diag{\sum_{k \in \sr{K}} \eta_{k} \vc{v}_{k}} + Q^{\vc{\theta}}$,
\eq{new A 2} can be written as
\begin{align}
\label{eq:new A 3}
  \widetilde{\sr{A}} \widetilde{f}(\vc{X}(t)) = \widetilde{f}(\vc{X}(t)) \left(\frac {[K^{\vc{\theta}}(\vc{\eta}) \widetilde{\vc{h}}]_{J(t)}} {\widetilde{h}(J(t))} + \sum_{k \in \sr{K}} \gamma_{k}(\vc{\eta}) y_{k}(t) \right).
\end{align}
Hence, $\widetilde{\sr{A}}$ has the exactly same form as $\sr{A}$ of \eq{A 3}, and therefore $\vc{X}(\cdot)$ under $\widetilde{\dd{P}}$ represents the $d$-dimensional MMFN with the same drift vectors $\vc{v}_{k}$ for $k \in \sr{K}$, the same reflecting matrix $R$ and the transition rate matrix $Q^{\vc{\theta}}$ for the background Markov chain.

Similar to $K(\vc{\theta})$, we denote the Perron-Frobenius eigenvalue and right and left eigenvectors of $K^{\vc{\theta}}(\vc{\eta})$ by $\gamma^{\vc{\theta}}(\vc{\eta})$, $\vc{h}^{\vc{\theta}}_{\vc{\eta}}$ and $(\vc{\xi}^{\vc{\theta}}_{\vc{\eta}})^{\rs{t}}$, respectively, and normalize $\vc{\xi}^{\vc{\theta}}_{\vc{\eta}}$ and $\vc{h}^{\vc{\theta}}_{\vc{\eta}}$ through $\vc{1}$. Namely, 
\begin{align}
\label{eq:nomalized xi1}
  K^{\vc{\theta}}(\vc{\eta}) \vc{h}^{\vc{\theta}}_{\vc{\eta}} = \gamma^{\vc{\theta}}(\vc{\eta}) \vc{h}^{\vc{\theta}}_{\vc{\eta}}, \qquad \br{\vc{\xi}^{\vc{\theta}}_{\vc{\eta}}, \vc{h}^{\vc{\theta}}_{\vc{\eta}}} = 1, \qquad \br{\vc{\xi}^{\vc{\theta}}_{\vc{\eta}}, \vc{1}} = 1.
\end{align}
Note that $\pi^{\vc{\theta}} \equiv \{\xi_{\vc{0}}^{\vc{\theta}}(i); i \in \sr{J}\}$ is the stationary distribution of $Q^{\vc{\theta}}$, where $\xi^{\vc{\theta}}_{\vc{0}}(i)$ is the $i$-th entry of $\vc{\xi}_{\vc{0}}^{\vc{\theta}}$. Then, applying $\widetilde{h}(i) = h_{\vc{\eta} + \vc{\theta}}(i) /h_{\vc{\theta}}(i)$ in \eq{new K 1} and using the fact that $K(\vc{\eta} + \vc{\theta}) \vc{h}_{\vc{\eta} + \vc{\theta}} = \gamma(\vc{\eta} + \vc{\theta}) \vc{h}_{\vc{\eta} + \vc{\theta}}$ we have
\begin{align*}
  \gamma(\vc{\eta} + \vc{\theta}) - \gamma(\vc{\theta}) = \frac{1}{\widetilde{h}(i)} \left(\sum_{k \in \sr{K}} \eta_{k} v_{k}(i) \widetilde{h}(i) + \sum_{j \in \sr{J}} q_{ij}^{\vc{\theta}} \widetilde{h}(j)\right) = \frac{1}{\widetilde{h}(i)} [K^{\vc{\theta}}(\vc{\eta}) \widetilde{\vc{h}}]_{i}.
\end{align*}
This implies that $\gamma(\vc{\eta} + \vc{\theta}) - \gamma(\vc{\theta})$ must be the Perron-Frobenius eigenvalues of $K^{\vc{\theta}}(\vc{\eta})$ because $\widetilde{\vc{h}}$ is a positive vector and $Q^{\vc{\theta}}$ is irreducible. Hence, $\widetilde{\vc{h}} = \vc{h}^{\vc{\theta}}_{\vc{\eta}}$ and
\begin{align}
\label{eq:new ga 1}
  \gamma^{\vc{\theta}}(\vc{\eta}) = \gamma(\vc{\eta} + \vc{\theta}) - \gamma(\vc{\theta}).
\end{align}
Let $\ol{\vc{v}}^{\vc{\theta}}$ be the mean drift vector whose $k$-th entry is $\sum_{i \in \sr{K}} v_{k}(i) \pi^{\vc{\theta}}(i)$, and let $\ol{\lambda}^{\vc{\theta}}_{k} = \sum_{i \in \sr{K}} \lambda_{k}(i) \pi^{\vc{\theta}}(i)$ and $\ol{\mu}^{\vc{\theta}}_{k} = \sum_{i \in \sr{K}} \mu_{k}(i) \pi^{\vc{\theta}}(i)$, then it follows from \eq{v k} that
\begin{align}
\label{eq:v k1}
  \ol{\vc{v}}^{\vc{\theta}} = \ol{\vc{\lambda}}^{\vc{\theta}} + P^{\rs{t}} \ol{\vc{\mu}}^{\vc{\theta}} - \ol{\vc{\mu}}^{\vc{\theta}},
\end{align}
while, similarly to \eq{mean v 1}, we have
\begin{align}
\label{eq:new mean v 1}
  \ol{\vc{v}}^{\vc{\theta}} = \nabla \gamma^{\vc{\theta}}(\vc{\eta})|_{\vc{\eta} = \vc{0}} = \nabla \gamma(\vc{\theta}).
\end{align}

We will use the following fact for deriving a lower bound of the tail decay rate.
\begin{lemma}\rm
\label{lem:stability 3}
For $k \in \sr{K}$, if $\ol{v}_{k}^{\vc{\theta}} < 0$, then station $k$ is stable under the change of measure by $E^{\vc{\theta}}(\cdot)$.
\end{lemma}
\begin{proof}
Similar to $\ol{\vc{\lambda}}^{\vc{\theta}}, \ol{\vc{\mu}}^{\vc{\theta}}$, define $\ol{\vc{\alpha}}^{\vc{\theta}}, \ol{\widetilde{\vc{\alpha}}}^{\vc{\theta}}$ for the solutions $\ol{\vc{\alpha}}, \ol{\widetilde{\vc{\alpha}}}$ of the linear and non-linear traffic equations, respectively. Then, $\ol{v}^{\vc{\theta}}_{k} < 0$ implies that
\begin{align}
\label{eq:}
  \ol{\widetilde{\alpha}}^{\vc{\theta}}_{k} - \ol{\mu}^{\vc{\theta}}_{k} & = \ol{\lambda}^{\vc{\theta}}_{k} + \sum_{\ell \in \sr{K}} \sum_{i \in \sr{I}} \min(\alpha_{\ell}(i), \mu_{\ell}(i)) \pi^{\vc{\theta}}(i) p_{\ell,k} - \ol{\mu}_{k} \nonumber\\
  & \le \ol{\lambda}^{\vc{\theta}}_{k} + \sum_{\ell \in \sr{K}} \sum_{i \in \sr{I}} \mu_{\ell}(i) \pi^{\vc{\theta}}(i) p_{\ell,k} - \ol{\mu}_{k} = \ol{v}^{\vc{\theta}}_{k} < 0,
\end{align}
where the last equality is obtained by \eq{v k1}. Hence, station $k$ is stable.
\end{proof} 

\subsection{Partially reflecting fluid network}
\label{sec:partially}

To fully utilize change of measure, we apply it to a $d$-dimensional Markov modulated fluid process whose $k$-th component is nonnegative and has a reflecting boundary at $0$ only for all $k \in A$ for a given $A \subset \sr{K}$, which is called an $S_{A}$-reflecting Markov modulated fluid network. To formally define this process, we introduce some notations for vectors and matrices. Recall the subvector and submatrix conventions (see the end of \sectn{dynamics}). Namely, $R_{A_{1},A_{2}}$ is the submatrix of $R$ whose row and column indexes are $A_{1}, A_{2}$, respectively. We here denote $R_{A,A}$ simply by $R_{A}$. For $A \subset \sr{K}$, we define the $S_{A}$-reflecting Markov modulated fluid network $\{\vc{Z}^{(A)}(t); t \ge 0\}$ as
\begin{align}
\label{eq:ZA 0}
  \vc{Z}^{(A)}(t) = \vc{Z}^{(A)}(0) + \vc{V}(t) + \left(\begin{array}{c} R_{A} \\ R_{A^{c},A} \end{array}\right) \vc{Y}^{(A)}_{A}(t) \ge \vc{0}, \qquad t \ge 0,
\end{align}
where $A^{c} = \sr{K} \setminus A$ and
\begin{align}
\label{eq:ZAY 2}
 \int_{0}^{t} Z^{(A)}_{k}(s) Y^{(A)}_{k}(ds) = 0, \quad Y^{(A)}_{k}(0) = 0, \quad \mbox{$Y^{(A)}_{k}(t) \uparrow$ in $t$}, \quad k \in A, t \ge 0.
\end{align}
These conditions are equivalent to $\vc{Z}(\cdot), \vc{B}(\cdot)$ satisfying \eq{Z 1} and \eq{b 1} in which $b_{k}(t)$ is replaced by $\mu_{k}(J(t))$ for $k \in A$, which shows that $S_{A}$ is indeed only a reflecting boundary of $\vc{Z}^{(A)}(t)$. Clearly, $\vc{Z}^{(\sr{K})}(\cdot) = \vc{Z}(\cdot)$, while $\vc{Z}^{(\emptyset)}(t) = \vc{Z}^{(\emptyset)}(0) + \vc{V}(t)$, which is a Markov additive process on $\dd{R}^{d}$.

It is not hard to see that the process $\vc{Z}^{(A)}(\cdot) \equiv \{\vc{Z}^{(A)}(t); t \ge 0\}$ uniquely exists similar to MMFN $\vc{Z}(\cdot)$. As a fluid network model, $\vc{Z}^{(A)}_{A}(\cdot)$ is influenced by stations in $A^{c}$ through $\vc{v}_{A}$, but it is completely determined by $Q$, $\vc{v}_{A}$ and $R_{A}$ as a MMFN on $\dd{R}_{+}^{A^{v}} \times \sr{J}$, while $\vc{Z}^{(A)}_{A^{c}}(\cdot)$ is influenced by $\vc{Z}^{(A)}_{A}(\cdot)$ through $\vc{Y}^{(A)}_{A}(\cdot)$.

Recall that $\vc{Z}^{(A)}(\cdot)$ is the $S_{A}$-reflecting fluid process. We note the following fact.

\begin{lemma}\rm
\label{lem:ZA limit 1}
Under the stability condition \eq{stability 1}, for $A \subset \sr{K}$, if $A \ne \sr{K}$, then
\begin{align}
\label{eq:ZA limit 1}
  \exists k \in A^{c}, \quad \lim_{t \to \infty} \vc{Z}^{(A)}_{k}(t) = -\infty, \qquad w.p.1.
\end{align}
\end{lemma}

This lemma is proved in \sectn{ZA limit 1}, and will be used in the proof of \lem{W 1}. Similar to $\vc{X}(\cdot)$, we define a Markov process for the $S_{A}$-reflecting MMFN for $A \subset \sr{K}$ by $\vc{X}^{(A)}(\cdot) \equiv \{(\vc{Z}^{(A)}(t),J(t)); t \ge 0\}$. Then, \eq{A 2} is changed to
\begin{align}
\label{eq:A 5}
 \sr{A}f(\vc{X}^{(A)}(t)) & = f(\vc{X}^{(A)}(t)) \left(\frac {[K(\vc{\theta}) \vc{h}]_{J(t)}} {h(J(t))} + \sum_{\ell \in A} \sum_{k \in \sr{K}} \theta_{k} r_{k,\ell} y^{(A)}_{\ell}(t) \right).
\end{align}

Similar to \eq{EM 1}, taking \eq{A 5} into account, we can define an exponential martingale $E^{\vc{\theta}}_{A}(t)$ for the Markov process $\vc{X}^{(A)}(\cdot)$ as
\begin{align}
 \label{eq:EM 2}
  E^{\vc{\theta}}_{A}(t) & = \frac {e^{\br{\vc{\theta},\vc{Z}^{(A)}(t)}} h_{\vc{\theta}}(J(t))} {e^{\br{\vc{\theta},\vc{Z}^{(A)}(0)}} h_{\vc{\theta}}(J(0))} \exp\left(- \gamma(\vc{\theta}) t - \sum_{k \in A} \gamma_{k}(\vc{\theta}) Y^{(A)}_{k}(t)\right), \qquad t \ge 0.
\end{align}

\subsection{Stationary tail probability}
\label{sec:tail}

Recall that we have defined face $A$ by $S_{A} = \cup_{k \in A} \{\vc{z} \in S; z_{k} = 0\}$ for nonempty $A \subset \sr{K}$. For $n \ge 1$, let $\tau_{A}^{\ex}(n)$ and $\tau_{A}^{\re}(n)$ be the $n$-th exit and return times from and to $S_{A}$. Namely, let $\tau_{A}^{\re}(0)=0$ and, for $n \ge 1$,
\begin{align*}
 & \tau_{A}^{\ex}(n) = \inf\{t > \tau_{A}^{\re}(n-1); \vc{Z}(t) \in S \setminus S_{A}\},\\
 & \tau_{A}^{\re}(n) = \inf\{t > \tau_{A}^{\ex}(n); \vc{Z}(t) \in S_{A}\}.
\end{align*}
Note that $\tau_{A}^{\ex}(1) = 0$ when $\vc{Z}(0) \in S \setminus S_{A}$. To consider the finiteness of these hitting times, we note the following simple fact, which is proved in \sectn{reachable 1}. 

\begin{lemma}
\label{lem:reachable 1}
Under the assumptions of \lem{stability 2}, we have, for each $k \in \sr{K}$, (a)
\begin{align}
\label{eq:reachable 1}
  \dd{P}(\exists t > 0, \vc{X}(t) \in S_{k} \times \sr{J}|\vc{X}(0) = \vc{x}) > 0, \qquad \forall \vc{x} \in S \times \sr{J},
\end{align}
(b) $\nu(S_{k} \times \sr{J}) > 0$ for the stationary distribution $\nu$ of $\vc{X}(t)$. 
 \end{lemma}

By this lemma and the existence of the stationary distribution, $\tau_{A}^{\ex}(n)$ and $\tau_{A}^{\re}(n)$ are finite $a.s.$ for all $n \ge 1$. Assume that $\vc{X}(0)$ is subject to the stationary distribution $\nu$, and define a point process $N^{\ex}_{A}$ if $\tau_{A}^{\ex}(n)$'s are finite as
\begin{align*}
  N^{\ex}_{A}(B) \equiv \sum_{n=1}^{\infty} 1(\tau_{A}^{\ex}(n) \in B), \qquad B \in \sr{B}(\dd{R}_{+}).
\end{align*}
Obviously, we have
\begin{align}
\label{eq:tau A1}
  N^{\ex}_{A}(B) \equiv \sum_{t \in B} 1(\vc{Z}(t-) \in S_{A}, \vc{Z}(t) \in S \setminus S_{A}), \qquad B \in \sr{B}(\dd{R}_{+}).
\end{align}

Let $\lambda^{\ex}_{A} \equiv \dd{E}_{\nu}(N^{\ex}_{A}((0,1]))$, then it must be positive and finite, and $N^{\ex}_{A}$ is a stationary point process jointly with $\vc{X}(\cdot)$, which has the finite intensity $\lambda^{\ex}_{A}$. Recall that we have extended process $\vc{X}(\cdot)$ on the whole line and introduced the shift operator group \{$\Theta_{t}; t \in \dd{R}\}$ on $\Omega$ in \sectn{asymptotic 1}. Then, the shift operator group operates on $N^{\ex}_{A}$ as
\begin{align*}
 (N^{\ex}_{A}(B) \circ \Theta_{t})(\omega) = N^{\ex}_{A}(B+t)(\omega),
\end{align*}
for $\omega \in \Omega$, $s, t \in \dd{R}$, $B \in \sr{B}(\dd{R})$, where $B + t = \{x+t; x \in B\}$ for $t \in \dd{R}$. Since $(\vc{X}(\cdot), N^{\ex}_{A})$ is jointly stationary, we define a Palm distribution $\dd{P}^{\ex}_{A}$ concerning $N^{\ex}_{A}$ on the measurable space $(\Omega,\sr{F})$ similar to \eq{Palm 1} as
\begin{align}
\label{eq:Palm 2}
  \dd{P}^{\ex}_{A}(D) = (\lambda^{\ex}_{A})^{-1} \dd{E}_{\nu}\left(\int_{0}^{1} 1_{D} \circ \Theta_{u} N^{\ex}_{A}(du)\right), \qquad D \in \sr{F},
\end{align}
where the index $0$ specifies that this Palm distribution has a unit mass at the origin. It is known that $\{\vc{X}(\tau_{A}^{\ex}(n)-); n \in \dd{Z}\}$ is stationary under $\dd{P}^{\ex}_{A}$. 

Hence, define
\begin{align*}
  \nu^{\ex}_{A}(B) = \dd{P}^{\ex}_{A}(\vc{X}(\tau_{A}^{\ex}(n)-) \in B), \qquad B \in \sr{B}(S) \times 2^{\sr{J}},
\end{align*}
then $\nu^{\ex}_{A}$ is the stationary distribution of $\vc{X}(\tau_{A}^{\ex}(n)-)$. Thus, the Palm inversion formula, which is also-called the cycle formula, yields
\begin{align}
\label{eq:cycle 1}
 & \dd{P}_{\nu}(\vc{X}(0) \in B) = \lambda^{\ex}_{A} \dd{E}_{\nu^{\ex}_{A}} \left( \int_{0}^{\tau_{A}^{\ex}(2)} 1(\vc{X}(u) \in B) du \right), \quad B \in \sr{B}(S) \times 2^{\sr{J}},
\end{align}
where we note that $\tau_{A}^{\ex}(1) = 0 < \tau_{A}^{\re}(1) < \tau_{A}^{\ex}(2)$ almost surely under $\dd{P}^{\ex}_{A}$. In what follows, we simply denote $\tau_{A}^{\re}(1)$ by $\tau_{A}^{\re}$. We note the following fact.
\begin{lemma}\rm
\label{lem:finite 1}
For nonempty $A \subset \sr{K}$ and $\vc{\theta} \in \dd{R}^{d}$, $\dd{E}_{\nu^{\ex}_{A}}(e^{\br{\vc{\theta}, \vc{Z}(0)}}) < \infty$ if and only if $\dd{E}_{\nu}(e^{\br{\vc{\theta}, \vc{Z}(0)}} 1(\vc{Z}(0) \in S_{A})) < \infty$.
\end{lemma}
\begin{proof}
Define $\Lambda^{\ex}_{A}(t)$ as
\begin{align*}
  \Lambda^{\ex}_{A}(t) = \sum_{j \in \sr{J} \setminus \{J(t-)\}} \sum_{k \in A} q_{J(t-),j} 1(\vc{Z}(t-) \in S_{A}, J(t)=j, a_{k}(t) > \mu_{k}(j)),
\end{align*}
then, from \eq{tau A1}, $\Lambda^{\ex}_{A}(t)$ is a stochastic intensity of $N^{\ex}_{A}(t)$, and therefore
\begin{align*}
  M_{A}(t) = N^{\ex}_{A}(t) - \int_{s}^{t} \Lambda^{\ex}_{A}(u) du, \qquad t \ge 0,
\end{align*}
is an $\dd{F}$-martingale. Hence, it follows from Papangelou formula (e.g., see Section 1.9.2 of \cite{BaccBrem2003}) that \begin{align*}
  \lambda^{\ex}_{A} \dd{E}_{\nu^{\ex}_{A}}(e^{\br{\vc{\theta}, \vc{Z}(0)}}) = \dd{E}_{\nu}(\Lambda^{\ex}_{A}(0) e^{\br{\vc{\theta}, \vc{Z}(0)}}).
\end{align*}
 Since $\Lambda^{\ex}_{A}(t)$ is bounded on $\vc{Z}(t-) \in S_{A}$, $\dd{E}_{\nu}(e^{\br{\vc{\theta}, \vc{Z}(0)}} 1(\vc{Z}(0) \in S_{A})) < \infty$ implies that $\dd{E}_{\nu^{\ex}_{A}}(e^{\br{\vc{\theta}, \vc{Z}(0)}}) < \infty$. For the opposite direction, define the random subset of $\sr{J}$ as
\begin{align}
\label{eq:q 1}
  \sr{J}^{\ex}_{A}(t) = \cup_{j \in \sr{J}} \cup_{k \in A} \{i \in \sr{J}; J(t)=j, q_{i,j} 1(a_{k}(t) > \mu_{k}(j)) > 0\}, \qquad t \in \dd{R},
\end{align}
then $\sr{J}^{\ex}_{A}(t) \ne \emptyset$ because $\vc{Z}(t)$ can not get out from $S_{A}$ otherwise. We now assume that $\dd{E}_{\nu}(1(\vc{Z}(0) \in S_{A}) e^{\br{\vc{\theta}, \vc{Z}(0)}}) = \infty$, then there is some $i \in \sr{J}$ such that
\begin{align*}
  \dd{E}_{\nu}(1(J(0-) = i, \vc{Z}(0-) \in S_{A}) e^{\br{\vc{\theta}, \vc{Z}(0)}})  = \dd{E}_{\nu}(1(J(0) = i, \vc{Z}(0) \in S_{A}) e^{\br{\vc{\theta}, \vc{Z}(0)}}) = \infty.
\end{align*}
If $i \in \sr{J}^{\ex}_{A}(t)$, then we have done because this together with \eq{q 1} imply that \linebreak $\dd{E}_{\nu}(\Lambda^{\ex}_{A}(0) e^{\br{\vc{\theta}, \vc{Z}(0)}}) = \infty$. If $i \not\in \sr{J}^{\ex}_{A}(t)$, then we can find a finite sequence of the state transitions of $J(t)$ from $i$ to $j \in \sr{J}^{\ex}_{A}(t)$ for $\vc{Z}(t)$ to stay in $S_{A}$ by the irreducibility of $J(t)$ and the fact that $Z(t)$ eventually gets out from $S_{A}$. Hence, $\dd{E}_{\nu}(\Lambda^{\ex}_{A}(0) e^{\br{\vc{\theta}, \vc{Z}(0)}}) = \infty$.
\end{proof} 

Let $C(x)$ be an element of $\sr{B}(\dd{R}_{+}^{d})$ for $x \ge 0$, then we apply change of measure to \eq{cycle 1}. For this, define $\tau_{x}$ as
\begin{align}
\label{eq:stopping 1}
  \tau_{x} = \inf\{t \ge 0; \vc{Z}(t) \in C(x)\}, \qquad x > 0.
\end{align}
Obviously, $\tau_{x}$ is an $\dd{F}$-stopping time. For $A \subset \sr{K}$, let
\begin{align*}
 & W_{A}(t) = \dd{E}_{\nu^{\ex}_{A}} \left(\left. \int_{t}^{\tau^{\re}_{A}} 1(\vc{Z}(u) \in C(x)) du \right|\sr{F}_{t} \right), 
\end{align*}
then it follows from \eq{EM 1} and \eq{cycle 1} that
\begin{align}
\label{eq:exp 1}
 &\dd{P}_{\nu}(\vc{Z}(0) \in C(x)) = \lambda^{\ex}_{A} \dd{E}_{\nu^{\ex}_{A}}(W_{A}(\tau_{x})1(\tau_{x} < \tau^{\re}_{A})) \nonumber\\
 & \quad = \lambda^{\ex}_{A} \widetilde{\dd{E}}^{\vc{\theta}}_{\nu^{\ex}_{A}} \Big[ \frac {h_{\vc{\theta}}(J(0))} {h_{\vc{\theta}}(J(\tau_{x}))} 1(\vc{Z}(0) \in S_{A}) e^{\br{\vc{\theta},\vc{Z}(0)}} 1( \tau_{x} < \tau^{\re}_{A}) W_{A}(\tau_{x}) \nonumber\\
 & \hspace{15ex} \times \exp\Big(- \br{\vc{\theta}, \vc{Z}(\tau_{x})}+{\gamma}(\vc{\theta}) \tau_{x} + \sum_{k \in \sr{K} \setminus A} \gamma_{k}(\vc{\theta}) Y_{k}(\tau_{x}) \Big) \Big],
\end{align}
because $Y_{k}(\tau_{x})=0$ for $k \in A$ if $\tau_{x} < \tau^{\re}_{A}$.

We will consider the following two cases for $C(x)$.
\begin{itemize}
\item [(\sect{lower}a)] Let $B_{0} \subset \dd{R}_{+}^{d}$ be a compact set, and let $\vc{c} \in \dd{R}_{+}^{d}$ be a unit vector, that is, $\br{\vc{c},\vc{c}} = 1$. For $x > 0$, define $C(x)$ as
\begin{align*}
   C(x) = x \vc{c} + B_{0} \equiv \{x \vc{c} + \vc{y}; \vc{y} \in B_{0}\}.
\end{align*}
Since $\vc{Z}(t)$ is continuous in $t$ and $\vc{Z}(\tau_{x}) - x \vc{c} \in B_{0}$, we have, for each $\vc{\theta} \in \dd{R}^{d}$, 
\begin{align*}
  \min\{\br{\vc{\theta},\vc{z}}, \vc{z} \in B_{0}\} \le \br{\vc{\theta}, \vc{Z}(\tau_{x})} - x \br{\vc{\theta}, \vc{c}} \le \max\{\br{\vc{\theta},\vc{z}}, \vc{z} \in B_{0}\},
\end{align*}
and therefore, for $\vc{z}, \vc{z}' \in \partial C(x)$,
\begin{align*}
  |\br{\vc{\theta}, \vc{Z}(\tau_{x}) - x\vc{c}}| & \le \max(\max\{\br{\vc{\theta},\vc{z}}, \vc{z} \in B_{0}\}, \max\{\br{\vc{\theta},-\vc{z}}, \vc{z} \in B_{0}\})\\
  & \le \|\vc{\theta}\| \max \{\|\vc{z}\|; \vc{z} \in B_{0}\},
\end{align*}
where we have used the Schwarz's inequality $|\br{\vc{\theta},\pm \vc{z}}| \le \|\vc{\theta}\| \|\vc{z}\|$. Hence, choosing $\delta = \max \{\|\vc{z}\|; \vc{z} \in B_{0}\}$, we have
\begin{align}
\label{eq:C 1}
  |\br{\vc{\theta},\vc{Z}(\tau_{x})} - x\br{\vc{\theta},\vc{c}}| < \delta \|\vc{\theta}\|.
\end{align}

\item [(\sect{lower}b)] For a unit vector $\vc{c} \ge 0$ and $x > 0$, let $C(x) = \{\vc{z} \in \dd{R}_{+}^{d}; \br{\vc{c}, \vc{z}} \ge x\}$. \\
In this case, $\br{\vc{c}, \vc{Z}(\tau_{x})} = x$, so we choose $\vc{\theta} = \alpha \vc{c}$ for $\alpha > 0$, then
\begin{align}
\label{eq:C 2}
  \br{\vc{\theta}, \vc{Z}(\tau_{x})} = \alpha x = x \br{\vc{\theta}, \vc{c}}.
\end{align}
\end{itemize}

We note the following fact, which is proved in \sectn{W 1}.
\begin{lemma}\rm
\label{lem:W 1}
For $C(x)$ of (\sect{lower}a) and (\sect{lower}b) and nonempty $A \subset \sr{K}$, if \eq{C 1} holds, then there are constants $c_{A}^{-}, c_{A}^{+} > 0$ such that
\begin{align}
\label{eq:W x 1}
  c_{A}^{-} < W_{A}(\tau_{x}) < c_{A}^{+} \quad \mbox{on } \tau_{x} < \tau^{\re}_{A}.
\end{align}
\end{lemma}

Due to this lemma, we can get upper and lower bounds for the tail decay rate from \eq{exp 1}. Since $1(\tau_{x} < \tau^{\re}_{A}) \le 1$, the upper bound is easily obtained, but the results are essentially the same as \thr{upper 1}, so we omit to discuss it. For the lower bound, we need to show that
\begin{align}
\label{eq:tau lower 1}
  \widetilde{\dd{P}}_{\nu^{\ex}_{A}} (\tau_{x} < \tau^{\re}_{A}) > 0.
\end{align}
This should be carefully considered, which will be done in \sectn{lower 1}.

\subsection{Proof of \thr{lower 1}}
\label{sec:lower 1}

Our starting point is \eq{exp 1}. In either $C(x) = x \vcn{e}_{k} + B_{0}$ for $x > 0$ or $C(x) = \{\vc{z} \in \dd{R}_{+}^{d}; \br{\vc{c}, \vc{z}} \ge x\}$, we have \eq{C 1}. Hence, by \lem{W 1}, \eq{exp 1} yields
\begin{align}
\label{eq:exp 2}
 e^{x \br{\vc{\theta}, \vc{c}} + \delta\|\vc{\theta}\|}\dd{P}_{\nu}(\vc{Z}(0) \in C(x)) \ge \lambda^{\ex}_{A} \widetilde{\dd{E}}^{\vc{\theta}}_{\nu^{\ex}_{A}} & \Big[ \frac {h_{\vc{\theta}}(J(0))} {h_{\vc{\theta}}(J(\tau_{x}))} 1(\vc{Z}(0) \in S_{A}) 1( \tau_{x} < \tau^{\re}_{A}) c_{A}^{-}\nonumber\\
 & \quad \times \exp\Big({\gamma}(\vc{\theta}) \tau_{x} + \sum_{k \in \sr{K} \setminus A} \gamma_{k}(\vc{\theta}) Y_{k}(\tau_{x}) \Big) \Big].
\end{align}
Since $h_{\vc{\theta}}(j)$ for $j \in \sr{J}$ are positive, we can get a finite and positive lower bound from \eq{exp 2} if the following conditions hold.
\begin{itemize}
\item [(\sect{lower}c)] \eq{tau lower 1} holds, that is, $\widetilde{\dd{P}}_{\nu^{\ex}_{A}} (\tau_{x} < \tau^{\re}_{A}) > 0$.
\item [(\sect{lower}d)] $\vc{\theta} \in \Gamma^{+}_{\sr{K} \setminus A}$, that is, $\gamma(\vc{\theta}) \ge 0$, $\gamma_{k}(\vc{\theta}) \ge 0$ for $k \in \sr{K} \setminus A$.
\end{itemize}

There is a trade off between conditions (\sect{lower}c) and (\sect{lower}d). Namely, if $A$ is larger, then (\sect{lower}c) is more restrictive because $\tau^{\re}_{A}$ is smaller, while (\sect{lower}d) is less restrictive. We now prove (a) and (b) separately.

\noindent (a) We choose $A = \{k\}$ for \eq{exp 2}. Then, (\sect{lower}d) is satisfied for $\vc{\theta} \in G_{k}$ by its definition. By \lem{stability 3}, stations $\ell \ne k$ are stable because $\ol{v}^{\vc{\theta}}_{\ell} = \frac {\partial} {\partial \theta_{\ell}} \gamma(\vc{\theta}) < 0$ for $\ell \ne k$. Hence, by \lem{stability 2}, $[R^{-1} \nabla \gamma(\vc{\theta})]_{k} > 0$ implies that station $k$ is unstable. Thus, under $\widetilde{\dd{P}}_{\nu^{\ex}_{\{k\}}}$, $Z_{k}(t)$ diverges as $t \to \infty$ while $Z_{\ell}(t)$'s for $\ell \ne k$ have the $(d-1)$-dimensional joint stationary distribution. This implies (\sect{lower}c). Thus, we have \eq{lower 1} from \eq{exp 2}.

\noindent (b) We choose $A = \sr{K}$ for \eq{exp 2}. So, $\sr{K} \setminus A = \emptyset$. Since $\vc{\theta} \in {\rm Corn}(\overleftarrow{\Gamma}^{-} \cap \partial \Gamma^{-})$, we have, by \eq{new mean v 1},
\begin{align*}
  \ol{\vc{v}}^{\vc{\theta}} = \nabla \gamma(\vc{\theta}) \ge 0.
\end{align*}
Hence, $R^{-1} \ol{\vc{v}}^{\vc{\theta}} \ge 0$, and therefore all stations are unstable under $\widetilde{\dd{P}}_{\nu^{\ex}_{\sr{K}}}$. This implies (\sect{lower}c), and we get \eq{lower 2}.

\section{Proofs of lemmas} \setnewcounter
\label{sec:proof}

\subsection{Proof of \lem{b 1}}
\label{sec:b 1}

\begin{proof}
From \eq{Z 1}, \eq{b 1} and the fact that $z_{k}(t) \ne 0$ implies $Z_{k}(t) \ne 0$, equivalently, $Z_{k}(t) > 0$, by \eq{b 2}, we have, for time $t$ at which $Z_{k}(t)$ and $B_{\ell}(t)$ are differentiable for all $k, \ell \in \sr{K}$. 
\begin{align*}
  z_{k}(t) & = (a_{k}(t) - b_{k}(t)) 1(z_{k}(t) \ne 0) = \left(a_{k}(t) - \mu_{k}(J(t))\right) 1(Z_{k}(t) > 0)\\
  & = \left(a_{k}(t) - \mu_{k}(J(t))\right) - \left(a_{k}(t) - \mu_{k}(J(t))\right)1(Z_{k}(t)=0).
\end{align*}
Differentiating both sides of \eq{Z 1}, we have that $z_{k}(t)=0$ implies $a_{k}(t) = b_{k}(t)$. Hence, $a_{k}(t) - \mu_{k}(J(t)) = b_{k}(t) - \mu_{k}(J(t)) \le 0$ on $Z_{k}(t) = 0$ by \eq{b 1}. Thus, we have that
\begin{align*}
  z_{k}(t) & = a_{k}(t) - \mu_{k}(J(t)) + 0 \vee \left(\mu_{k}(J(t)) - a_{k}(t)\right)1(Z_{k}(t)=0)\\
  & = a_{k}(t) - \mu_{k}(J(t))1(Z_{k}(t)>0) + [(- \mu_{k}(J(t)) \vee (-a_{k}(t))]1(Z_{k}(t)=0)\\
  & = a_{k}(t) - [\mu_{k}(J(t))1(Z_{k}(t)>0) + (\mu_{k}(J(t)) \wedge a_{k}(t)) 1(Z_{k}(t)=0)],
\end{align*}
where $u \vee v = \max(u,v)$ and $u \wedge v = \min(u,v)$ for $u,v \in \dd{R}$. This and $z_{k}(t) = a_{k}(t) - b_{k}(t)$ conclude the equality in \eq{b 3}. It remains to prove that $b_{k}(t) \ge 0$. Let $\sr{K}_{0}(t) = \{k \in \sr{K}; Z_{k}(t) = 0\}$, then $b_{k}(t) = \mu_{k}(J(t)) \ge 0$ for $k \in \sr{K} \setminus \sr{K}_{0}$ by \eq{b 3}. On the other hand, on $Z_{k}(t)=0$, $a_{k}(t) - b_{k}(t) = z_{k}(t) = 0$ implies that,
\begin{align*}
  0 = \lambda_{k}(J(t)) + \sum_{\ell \in \sr{K}} b_{\ell}(t) p_{\ell,k} - b_{k}(t) \ge \sum_{\ell \in \sr{K}_{0}} b_{\ell}(t) p_{\ell,k} - b_{k}(t), \qquad k \in \sr{K}_{0}.
\end{align*}
Thus, we have $\vc{b}_{\sr{K}_{0}}(t) \ge P_{\sr{K}_{0}, \sr{K}_{0}} \vc{b}_{\sr{K}_{0}}(t)$, where $\vc{b}_{\sr{K}_{0}}(t)$ is $\sr{K}_{0}$-dimensional vector whose $\ell$-th entry is $b_{\ell}(t)$ for $\ell \in \sr{K}_{0}$, and $P_{\sr{K}_{0}, \sr{K}_{0}}$ is the $\sr{K}_{0} \times \sr{K}_{0}$ submatrix of $P$. Since $P$ and therefore $P_{\sr{K}_{0}, \sr{K}_{0}}$ are strictly substochastic, $\vc{b}_{\sr{K}_{0}}(t) \ge (P_{\sr{K}_{0}, \sr{K}_{0}})^{n} \vc{b}_{\sr{K}_{0}}(t) \to 0$ as $n \to \infty$. Thus, we have that $b_{k}(t) \ge 0$ for $k \in \sr{K}_{0}$, which completes the proof.
\end{proof}

\subsection{Proof of \lem{Dynkin 1}}
\label{sec:Dynkin 1}

Define $M(t)$ for $t \ge 0$ as
\begin{align*}
  M(t) & = \int_{0}^{t} g(\vc{Z}(s)) [h(J(s)) - \dd{E}\left(\left. h(J(s)) \right| \sr{F}_{s-}\right)] dN(s)\\
  & \quad + \int_{0}^{t} g(\vc{Z}(s)) \dd{E}\left(\left. \Delta h(J(s)) \right| \sr{F}_{s-}\right)] (dN(s) + q_{J(s-),J(s-)} ds),
\end{align*}
then $M(t)$ is an $\sr{F}_{t}$-martingale because $\vc{Z}(t)$ is continuous in $t$ because $- q_{J(s-),J(s-)} ds$ is the compensator of $d N(s)$ and $\dd{E}(|M(t)|) < \infty$ for each $t \ge 0$. Hence,
\begin{align}
\label{eq:jump 1}
 & \int_{0}^{t} g(\vc{Z}(s)) \Delta h(J(s)) dN(s) \nonumber\\
 & \quad = \int_{0}^{t} g(\vc{Z}(s)) \left[ h(J(s)) - \dd{E}\left(\left. h(J(s)) \right| \sr{F}_{s-}\right) + \dd{E}\left(\left. \Delta h(J(s)) \right| \sr{F}_{s-}\right) \right] dN(s) \nonumber\\
 & \quad = \int_{0}^{t} g(\vc{Z}(s)) \dd{E}\left(\left. \Delta h(J(s)) \right| \sr{F}_{s-}\right) (-q_{J(s-),J(s)}) ds + M(t) \nonumber\\
 & \quad = \int_{0}^{t} g(\vc{Z}(s)) \Bigg( \sum_{j \in \sr{J}} q_{J(s-),j} h(j) \Bigg) ds + M(t).
\end{align}
Since $Y_{\ell}(s)$ has the right-hand derivative $y_{\ell}(s)$ for $s \ge 0$, \eq{evolution 1} and \eq{jump 1} yield \eq{Dynkin 1}.

\subsection{Proof of \lem{gamma 1}}
\label{sec:gamma 1}

 (a) The convexity of $\gamma(\vc{\theta})$ is immediate from Theorem of \citet{Cohe1981} (see also Corollary 1.1 of \cite{Nuss1986}) since $K(\vc{\theta}) + aI$ is essentially nonnegative and its diagonal entries are linear functions of $\vc{\theta}$.\\
  (b) We note that $\gamma(\vc{\theta})$ is the root of the characteristic equation of $x \in \dd{C}$,
\begin{align*}
  | x I - K(\vc{\theta}))| = 0,
\end{align*}
whose real part is the largest among the roots of this equation. We also know that $\gamma(\vc{\theta})$ is unique as a Perron-Frobenius eigenvalue. Since $| x I - K(\vc{z}))|$ is a monic polynomial with order $m$ for complex $\vc{z} \in \dd{C}^{d}$, $x = \gamma(\vc{z})$ is at most an $m$-multiple valued analytic function for $\vc{z} \in \dd{C}^{d}$ (e.g., see Theorem 9.21 of \cite{Rudi1976}). On the other hand, $\gamma(\vc{\theta})$ is unique as a function on $\dd{R}^{d}$, and therefore  $\gamma(\vc{z})$ must be a single valued analytic on a neighborhood around the real axis in $\dd{C}^{d}$. Thus, $\gamma(\vc{\theta})$ is continuously partially differentiable with respect to $\theta_{i}$ for $\vc{\theta} \in \dd{R}^{d}$. Since $\vc{h}_{\vc{\theta}}$ is unique by the normalization \eq{normalize h1}, it can be obtained as a unique solution of the linear equations \eq{K 1} and \eq{normalize h1}. Hence, $\vc{h}_{\vc{\theta}}$ is obtained by Cramer's rule, and therefore it is continuously partially differentiable in entry-wise as long as $\gamma(\vc{\theta})$ does so. These facts complete the proof of (b).\\
(c) Premultiplying both sides of \eq{K 1} by $\vc{\xi}_{\vc{\theta}}$, we have, by \eq{normalize h1},
\begin{align*}
  \vc{\xi}_{\vc{\theta}} K(\vc{\theta}) \vc{h}_{\vc{\theta}} = \gamma(\vc{\theta})
\end{align*}
 Partially differentiating both sides of this formula concerning $\theta_{k}$ for $k \in \sr{K}$ in entry-wise, we have
\begin{align*}
 & \left( \nabla \xi_{\vc{\theta}}(1), \ldots, \nabla \xi_{\vc{\theta}}(m) \right) K(\vc{\theta}) \vc{h}_{\vc{\theta}} + \vc{\xi}_{\vc{\theta}} \left(\begin{array}{ccc} \hspace{-1ex} \diag{\vc{v}_{1}}, \ldots, \diag{\vc{v}_{d}} \hspace{-1ex} \end{array}\right) \vc{h}_{\vc{\theta}} \\
  & \qquad + \vc{\xi}_{\vc{\theta}} K(\vc{\theta}) \left(\begin{array}{c} \hspace{-1ex} \nabla h_{\vc{\theta}}(1)\\ \vdots\\ \nabla h_{\vc{\theta}}(m) \hspace{-1ex} \end{array}\right) = \left( \left(\frac {\partial} {\partial \theta_{1}} {\gamma}(\vc{\theta}) \right), \ldots, \left( \frac {\partial} {\partial \theta_{d}} {\gamma}(\vc{\theta})\right) \right) ,
\end{align*}
Since $K(\vc{0}) \vc{h}_{\vc{0}} = \vc{0}$, $\vc{\xi}_{\vc{0}} = \vc{\pi}$ and $\vc{\xi}_{\vc{0}} K(\vc{0}) = \vc{0}$, we have \eq{mean v 1}. 

\subsection{Proof of \lem{convex 1}}
\label{sec:convex 1}

  (a) This is immediate from (a) of \lem{gamma 1} since $\gamma(\vc{\theta})$ is finite for all $\vc{\theta} \in \dd{R}^{d}$ and convex. 
  
\noindent  (b) Since $\ol{\vc{v}} = \nabla \gamma(\vc{\theta})|_{\vc{\theta}=\vc{0}}$ by \lem{gamma 1}, $H_{\ol{\vc{v}}}$ contacts $\Gamma^{-}$ at the origin. Since $\Gamma^{-}$ is a convex set, $\Gamma^{-}$ is supported by $\{\vc{\theta} \in \dd{R}^{d}; \br{\ol{\vc{v}}, \vc{\theta}} \le 0\}$. 

\noindent (c) Under the stability condition \eq{stability 1}, there is at least one $k \in \sr{K}$ such that $\ol{v}_{k} < 0$, that is, $\br{\ol{\vc{v}}, \vc{e}_{k}} < 0$ because $\ol{\vc{v}} \ge \vc{0}$ implies $R^{-1} \ol{\vc{v}} \ge \vc{0}$, which is impossible by \eq{stability 1}. This implies the claim since $\ol{\vc{v}} = \nabla \gamma(\vc{\theta})|_{\vc{\theta}=\vc{0}}$, 
  
 \noindent (d) Let $\vc{u}_{k}$ be the $k$-th row of $R^{-1}$, then $\vc{u}_{k} R^{\ell} = 1(k = \ell)$ for $k, \ell \in \sr{K}$ by $R^{-1} R = I$, where $R^{\ell}$ is the $\ell$-th column of $R$. Since $\{\vc{u}_{k}; k \in \sr{K}\}$ is the base of the vector space $\dd{R}^{d}$, we can write $\vc{\theta} \in \dd{R}^{d}$ as
\begin{align*}
  \vc{\theta} = \sum_{k \in \sr{K}} a_{k} \vc{u}_{k}, \qquad \exists \vc{a} \equiv (a_{1}, \ldots, a_{d}) \in \dd{R}^{d}.
\end{align*}
We first take $\vc{a}$ such that  $a_{k} > 0$ and $a_{\ell} = 0$ for all $\ell \in \sr{K} \setminus \{k\}$, then $\vc{\theta} R^{\ell} = 0$ and therefore $\vc{\theta} \in H_{\ell}$ for $\ell \ne k$, so $\vc{\theta} \in \sr{R}_{k}$. Since $\alpha_{k} < \mu_{k}$ if and only if $\br{\vc{u}_{k}, \ol{\vc{v}}} = [R^{-1} \ol{\vc{v}}]_{k} < 0$, equivalently, $\br{\vc{\theta}, \ol{\vc{v}}} = a_{k} \br{\vc{u}_{k}, \ol{\vc{v}}} < 0$, which occurs only when $\sr{R}_{k}$ intersects $\Gamma^{-}$.

\subsection{Proof of \lem{Palm 1}}
\label{sec:Palm 1}

(a) For each $n \ge 1$, define the truncation function $w_{n}(x)$ from $\dd{R}_{+}$ to $\dd{R}$ as
\begin{align}
\label{eq:wn 1}
  w_{n}(x) & = x 1(x \le n) + \frac {2x - (x-n)^{2}} {2} 1(n < x \le n+1) + \frac {2n+1}{2} 1(n+1<x),
\end{align}
which is bounded by $n + \frac 12$, and has the bounded and continuous derivative:
\begin{align*}
  w_{n}'(x) = 1(x \le n) + (n+1 - x) 1(n < x \le n+1) \in [0,1].
\end{align*}
For each $n \ge 1$ and $k \in \sr{K}$, we choose test function $f(\vc{x}) = w_{n}(z_{k})$ for the Dynkin's formula \eq{Dynkin 1}, which means that $h(j) = 1$ for all $j \in \sr{J}$ in the formula \eq{A 1}. We then take the expectation of \eq{Dynkin 1} with $t=1$ under the stationary distribution $\nu$, which yields
\begin{align*}
  \dd{E}_{\nu}\left(\int_{0}^{1} w_{n}'(Z_{k}(s)) v_{k}(s) ds \right) + \sum_{\ell \in \sr{K}} r_{k,\ell} \dd{E}_{\nu} \left(\int_{0}^{1} w_{n}'(Z_{k}(s)) y_{\ell}(s) ds\right) = 0, \qquad k \in \sr{K}.
\end{align*}
Since $v_{k}(s)$ and $y_{\ell}(s)$ are bounded by constants and $\lim_{n  \to \infty} w_{n}'(x) = 1$, letting $n \to \infty$ in the above formula yields
\begin{align*}
  \dd{E}_{\nu}(V_{k}(1)) + \sum_{\ell \in \sr{K}} r_{k,\ell} \dd{E}_{\nu} (Y_{\ell}(1))= 0, \qquad k \in \sr{K}.
\end{align*}
This proves that $\dd{E}_{\nu}(Y_{k}(1)) = - \dd{E}_{\nu}([R^{-1} \vc{V}(1)]_{k}) > 0$ by \eq{stability 1}.\\
(b) If $\dd{P}_{\nu}(\vc{Z}(0) \in S_{k}) = 0$, then 
\begin{align*}
  \dd{E}_{\nu}\Big(\int_{0}^{1} 1(Z_{k}(u) = 0) du \Big) = \dd{P}_{\nu}(Z_{k}(0) = 0) = \dd{P}_{\nu}(\vc{Z}(0) \in S_{k}) = 0.
\end{align*}
Hence, $Y_{k}(1) = 0$ $a.s.$ $\dd{P}_{\nu}$. This contradicts (a), and therefore $\dd{P}_{\nu}(\vc{Z}(0) \in S_{k}) > 0$.

\subsection{Proof of \lem{upper 1}}
\label{sec:lem-upper 1}

(a) To prove \eq{mgf 2}, we only need to show that it is impossible to have that
\begin{align}
\label{eq:mgf 3}
  \limsup_{x \to \infty} \frac 1x \log \dd{P}_{\nu}(\br{\vc{c},\vc{Z}} > x) < - \sup\{\alpha \ge 0, \alpha \vc{c} \in \sr{D}\},
\end{align}
because the left-hand side of \eq{mgf 3} is bounded by the right-hand side by \eq{mgf 1}. Suppose that \eq{mgf 3} holds. Then, we can find $\alpha_{0} > \sup\{\alpha \ge 0, \alpha \vc{c} \in \sr{D}\}$ such that the left-hand side of \eq{mgf 3} is bounded by $-\alpha_{0}$. This $\alpha_{0}$ must satisfy that $\alpha_{0} \vc{c} \not\in \sr{D}$, and therefore $\varphi(\alpha_{0}\vc{c}) = \infty$. On the other hand, for $x \ge 0$,
\begin{align*}
  \int_{0}^{x} e^{\alpha_{0} u} d\dd{P}(\br{\vc{c},\vc{Z}} \le u) + e^{\alpha_{0} x} \dd{P}(\br{\vc{c},\vc{Z}} > x) = \dd{P}(\br{\vc{c},\vc{Z}} > 0) + \alpha_{0} \int_{0}^{x} e^{\alpha_{0} u} \dd{P}(\br{\vc{c},\vc{Z}} > u) du.
\end{align*}
Hence, $\varphi(\alpha_{0} \vc{c}) = \infty$ implies that the left side of this equation diverges as $x \to \infty$, and therefore
\begin{align*}
  \int_{0}^{\infty} e^{\alpha_{0} u} \dd{P}(\br{\vc{c},\vc{Z}} > u) du = \infty,
\end{align*}
which implies that $\limsup_{x \to \infty} \frac 1x \log \dd{P}_{\nu}(\br{\vc{c},\vc{Z}} > x) < -\alpha_{0}$ is impossible. Thus, \eq{mgf 3} is impossible.\\
(b) Fix a direction vector $\vc{c} \in U_{+}^{d}$. Since $B_{0}$ is compact, we can find some $\delta > 0$ for this $\vc{c}$ such that $\vc{Z} \in x \vc{c} + B_{0}$ implies that $\br{\vc{\theta}, \vc{Z}} \ge x \br{\vc{\theta}, \vc{c}} - \max_{\|\vc{y}\| \le \delta} \br{\vc{\theta},\vc{y}}$ for $\vc{\theta} \in \dd{R}^{d}$. Hence, 
\begin{align*}
  e^{x \br{\vc{\theta},\vc{c}}} \dd{P}(\vc{Z} \in x \vc{c} + B_{0}) & \le e^{x\br{\vc{\theta},\vc{c}}} \dd{P}\left(\br{\vc{\theta}, \vc{Z}} \ge x \br{\vc{\theta}, \vc{c}} - \max_{\|\vc{y}\| \le \delta} \br{\vc{\theta},\vc{y}}\right)\\
  &  \le e^{\max_{\|\vc{y}\| \le \delta} \br{\vc{\theta},\vc{y}}} \dd{E}(e^{x \br{\vc{\theta},\vc{Z}}}),
\end{align*}
and therefore, for $\vc{\theta} \in \sr{D}$,
\begin{align*}
  \limsup_{x \to \infty} \frac 1x \log \dd{P}(\vc{Z} \in x \vc{c} + B_{0}) \le - \br{\vc{\theta},\vc{c}}
\end{align*}
Taking the infimum of this inequality for all $\vc{\theta} \in \sr{D}$ yields \eq{upper bound 1}.

\subsection{Proof of \lem{fixed 1}}
\label{sec:fixed 1}

We first note that (c) is immediate from \lem{domain 2d}. Using this fact, we prove (a). The case for $d=2$ is again immediate from (c).  Thus, we prove (a) for $d \ge 3$. For this, we choose $A = \{\ell\}$ or $A = \{\ell'\}$ for $\ell, \ell'\in \sr{K}$ such that $\ell \ne \ell'$. For $\vc{\theta} \in \dd{R}^{d}$, let $\vc{\theta}^{\ell,\ell'}$ be the vector whose $\ell, \ell'$ entries are $\theta_{\ell}, \theta_{\ell'}$ while all the other entries vanish. Then, for $k' \ne \ell, \ell'$,
\begin{align}
\label{eq:nonpositive 1}
   \gamma_{k'}(\vc{\theta}^{\ell,\ell'}) = - p_{k',\ell} \theta_{\ell} - p_{k',\ell'} \theta_{\ell'} \le 0, \qquad \vc{\theta}^{\ell,\ell'} \ge 0.
\end{align}
 Assume that $\varphi_{k}(\vc{\theta}^{\ell,\ell'}) < \infty$, equivalently, $\psi_{k}(\vc{\theta}^{\ell,\ell'}) < \infty$, for $k = \ell, \ell'$, then, by (b) of \lem{marginal 2}, \eq{BAR 2} for $A = \{\ell\}$ can be written as
\begin{align*}
  (-\gamma(\vc{\theta}^{\ell,\ell'})) \psi(\vc{\theta}^{\ell,\ell'}) & + \sum_{k' \ne \ell,\ell'} (-\gamma_{k'}(\vc{\theta}^{\ell,\ell'})) \psi_{k'}(\vc{\theta}^{\ell,\ell'}) \nonumber \\
& + (-\gamma_{\ell'}(\vc{\theta}^{\ell,\ell'})) \psi_{\ell'}(\vc{\theta}^{\ell,\ell'}) \le \gamma_{\ell}(\vc{\theta}^{\ell,\ell'}) \psi_{\ell}(\vc{\theta}^{\ell,\ell'}).
\end{align*}
Similarly, we have, for $A = \{\ell'\}$,
\begin{align*}
  (-\gamma(\vc{\theta}^{\ell,\ell'})) \psi(\vc{\theta}^{\ell,\ell'}) & + \sum_{k' \ne \ell,\ell'} (-\gamma_{k'}(\vc{\theta}^{\ell,\ell'})) \psi_{k'}(\vc{\theta}^{\ell,\ell'}) \nonumber \\
& + (-\gamma_{\ell}(\vc{\theta}^{\ell,\ell'})) \psi_{\ell}(\vc{\theta}^{\ell,\ell'}) \le \gamma_{\ell'}(\vc{\theta}^{\ell,\ell'}) \psi_{\ell'}(\vc{\theta}^{\ell,\ell'}).
\end{align*}
Because of \eq{nonpositive 1} and the fact that the rays $H_{\ell}, H_{\ell'}$ on $\dd{R}^{d}_{\ell,\ell'}$ are in $\{\vc{\theta} \in \dd{R}^{d}_{\ell,\ell'}; \vc{\theta} \ge \vc{0}\}$ or $\{\vc{\theta} \in \dd{R}^{d}_{\ell,\ell'}; \vc{\theta} \le \vc{0}\}$ (see (d) of \lem{convex 1} for the definition of a ray), these two equations can be considered as the two dimensional fixed point equation with $\gamma(\vc{\theta}^{\ell,\ell'}), \gamma_{\ell}(\vc{\theta}^{\ell,\ell'}), \gamma_{\ell}(\vc{\theta}^{\ell,\ell'})$ for $\psi(\vc{\theta}^{\ell,\ell'})$. Hence, by (c) of this lemma, we can see that, for $D^{(0)}_{k} = \overleftarrow{\vc{0}}_{\sr{K} \setminus \{k\}}$ for $k \in \sr{K}$, $D^{(1)}_{k'} \ne \overleftarrow{\vc{0}}_{\sr{K} \setminus \{k'\}}$ for $k' = \ell, \ell'$. This proves (a).

\noindent (b) By (a), there is a nontrivial $D^{(0)}_{k}$ for $k \in \sr{K}$ satisfying $D^{(0)}_{k} \subset D^{(1)}_{k} \subset \widetilde{\sr{D}}_{k}$. Hence, we inductively see that $D^{(n-1)}_{k} \subset D^{(n)}_{k} \subset \widetilde{\sr{D}}_{k}$ for $n \ge 1$ and $k \in \sr{K}$. This implies that the limit $D^{(\infty)}_{k} \equiv \lim_{n \to \infty} D^{(n)}_{k} \subset \widetilde{\sr{D}}_{k}$ exists and $\{D^{(\infty)}_{k}; k \in \sr{K}\}$ is the solution \eq{fixed 1}. We now collect all the solutions of \eq{fixed 1}, and index them as $\{D_{k}(\beta); \in \sr{K}\}$ by $\beta \in \exists \sr{B}$. Define, 
\begin{align*}
  D^{(0)}_{k}(\sr{B}) = \cup_{\beta \in \sr{B}} D_{k}(\beta), \qquad k \in \sr{K}, \beta \in \sr{B},
\end{align*}
then $D^{(0)}_{k}(\sr{B}) \subset \sr{D}_{k}$ for $\forall k\in \sr{K}$. Since each $\{D_{k}(\beta); k \in \sr{K}\}$ is the solution of \eq{fixed 1} for each $\beta \in \sr{B}$, $D^{(1)}_{k}(\sr{B})$ by the iteration of \eq{fixed 2} must includes $D^{(0)}_{k}(\sr{B})$ for each $k \in \sr{K}$. Thus, we have the limit of $D^{(n)}_{k}(\sr{B})$ as $n \to \infty$. We denote this limit by $D^{(\max)}_{k}$, then it must be maximal by its construction. This proves (b).

\subsection{Proof of \lem{marginal 2}}
\label{sec:marginal 2}

(a) By (b) of \lem{marginal 1}, $\varphi(\vc{\theta}) < \infty$ is equivalent to $\psi(\vc{\theta}) < \infty$, which implies that $\psi_{k}(\vc{\theta}) < \infty$ for $k \in \sr{K}$. To derive \eq{BAR 1}, we truncate the test function $f_{\vc{\theta}}$ using $w_{n}$ of \eq{wn 1} for each $n \ge 1$. Namely, we put $f(\vc{x}) = e^{w_{n}(\br{\vc{\theta},\vc{z}})} h(j)$ for $h \in F_{+}(\sr{J})$, and apply this test function $f$ to \eq{Dynkin 1} with $t=1$. Since $e^{w_{n}(\br{\vc{\theta},\vc{z}})}$ is bounded in $\vc{z} \in \dd{R}^{d}$, we have, for each $n \ge 1$ and random vector $(\vc{Z},J)$ subject to $\nu$,
\begin{align}
\label{eq:n-BAR 1}
   \sum_{k \in \sr{K}} \dd{E}_{\nu} \Bigg( e^{w_{n}(\br{\vc{\theta},\vc{Z}}} \Big(& w_{n}'(\br{\vc{\theta},\vc{Z}}) \theta_{k} v_{k}(J) h(J) + \sum_{j \in \sr{J}} q_{J,j} h(j)\Big)\Bigg) \nonumber\\
  & + \sum_{\ell,k \in \sr{K}} \gamma(\vc{\theta}) m_{\ell} \dd{E}_{\ell}\left(w_{n}'(\br{\vc{\theta},\vc{Z}}) e^{w_{n}(\br{\vc{\theta},\vc{Z}})} \theta_{k} r_{k,\ell} h(J) \right) = 0.
\end{align}
Similar to $K(\vc{\theta})$, define $m \times m$ matrix $K_{n}(\vc{\theta})$ as
\begin{align*}
  K_{n}(\vc{\theta}) = w_{n}'(\br{\vc{\theta},\vc{Z}}) \diag{\sum_{k = 1}^{d} \theta_{k} \vc{v}_{k}} + Q,
\end{align*}
then it has a unique Perron Frobenius eigenvalue and eigenvector, which are denoted by $\gamma^{(n)}(\vc{\theta})$ and $\vc{h}^{(n)}_{\vc{\theta}}$, respectively, because $K_{n}(\vc{\theta})$ is an irreducible $\sr{M}$-matrix. Hence, similarly to the derivation of \eq{A 3} from \eq{A 2}, \eq{n-BAR 1} is written as
\begin{align}
\label{eq:n-BAR 2}
   \gamma^{(n)}(\vc{\theta}) \dd{E}_{\nu} \Bigg( e^{w_{n}(\br{\vc{\theta},\vc{Z}}} h^{(n)}_{\vc{\theta}}(J) \Bigg)  + \sum_{\ell \in \sr{K}} \gamma_{\ell}(\vc{\theta}) m_{\ell} \dd{E}_{\ell}\left(w_{n}'(\br{\vc{\theta},\vc{Z}}) e^{w_{n}(\br{\vc{\theta},\vc{Z}}} h^{(n)}_{\vc{\theta}}(J) \right) = 0.
\end{align}
Since $w_{n}'(x)$ converges to $1$ as $n \to \infty$, $K_{n}(\vc{\theta})$ converges to $K(\vc{\theta})$ in entry-wise as $n \to \infty$. Then, one can see that $\gamma^{(n)}(\vc{\theta})$ and $\vc{h}^{(n)}_{\vc{\theta}}$ converge to $\gamma(\vc{\theta})$ and $\vc{h}_{\vc{\theta}}$, respectively, as $n \to \infty$ (e.g., see Theorem 3.1 of \cite{Meye2015}). Hence, the bounded convergence theorem and \eq{n-BAR 2} yield \eq{BAR 1} because $\psi(\vc{\theta}) < \infty$ and $\psi_{k}(\vc{\theta}) < \infty$ for $\forall k \in \sr{K}$.

\noindent (b) We start with \eq{n-BAR 2}, and rearrange it as
\begin{align}
\label{eq:n-BAR 3}
   (- \gamma^{(n)}(\vc{\theta})) \dd{E}_{\nu} \left( e^{w_{n}(\br{\vc{\theta},\vc{Z}})} h^{(n)}_{\vc{\theta}}(J) \right) & + \sum_{\ell \in \sr{K} \setminus A} (-\gamma_{\ell}(\vc{\theta})) m_{\ell} \dd{E}_{\ell}\left(w_{n}'(\br{\vc{\theta},\vc{Z}}) e^{w_{n}(\br{\vc{\theta},\vc{Z}})} h^{(n)}_{\vc{\theta}}(J) \right) \nonumber\\
   & = \sum_{\ell \in A} \gamma_{\ell}(\vc{\theta}) m_{\ell} \dd{E}_{\ell}\left(w_{n}'(\br{\vc{\theta},\vc{Z}}) e^{w_{n}(\br{\vc{\theta},\vc{Z}}} h^{(n)}_{\vc{\theta}}(J) \right).
\end{align}
The left-hand side of this equation is the sum of nonnegative terms with nonnegative coefficients, $-\gamma^{(n)}(\vc{\theta}), - \gamma_{\ell}(\vc{\theta})$ for sufficiently large $n$ because we can find $n_{0} \ge 1$ such that $\gamma^{(n)}(\vc{\theta}) < 0$ for all $n \ge n_{0}$ due to $\gamma(\vc{\theta}) < 0$ by the assumption. On the other hand, its right-hand side is uniformly bounded by $\psi_{k}(\vc{\theta})<\infty$ for $k \in A$ because $0 \le w_{n}'(x) \le 1$. Hence, we have \eq{BAR 2} because letting $n \to \infty$ in \eq{n-BAR 3} and applying Fatou's lemma yield
\begin{align*}
 & \dd{E}_{\nu} \left( e^{\br{\vc{\theta},\vc{Z}}} h_{\vc{\theta}}(J) \right) = \dd{E}_{\nu} \left( \liminf_{n \to \infty} e^{w_{n}(\br{\vc{\theta},\vc{Z}})} h^{(n)}_{\vc{\theta}}(J) \right) \le \liminf_{n \to \infty} \dd{E}_{\nu} \left( e^{w_{n}(\br{\vc{\theta},\vc{Z}})} h^{(n)}_{\vc{\theta}}(J) \right),\\
 & \dd{E}_{\ell}\left(e^{\br{\vc{\theta},\vc{Z}}} h_{\vc{\theta}}(J) \right) \le \dd{E}_{\ell}\left( \liminf_{n \to \infty} w_{n}'(\br{\vc{\theta},\vc{Z}}) e^{w_{n}(\br{\vc{\theta},\vc{Z}})} h^{(n)}_{\vc{\theta}}(J) \right)\\
 & \hspace{18ex} \le \liminf_{n \to \infty} \dd{E}_{\ell}\left(w_{n}'(\br{\vc{\theta},\vc{Z}}) e^{w_{n}(\br{\vc{\theta},\vc{Z}})} h^{(n)}_{\vc{\theta}}(J) \right).
\end{align*}
Obviously, \eq{BAR 2} show that $\varphi(\vc{\theta})$ and $\varphi_{k}(\vc{\theta})$ are finite under the assumptions in (b), which further implies \eq{BAR 1} by (a). Thus, (b) of this lemma is proved.

\subsection{Proof of \lem{ZA limit 1}}
\label{sec:ZA limit 1}

For $A \subset \sr{K}$, define the fluid scaling limits of $\vc{V}^{A}(\cdot)$, $\vc{Z}^{A}(\cdot)$ and $\vc{Y}^{A}(\cdot)$ as
\begin{align*}
  \ol{\vc{V}}^{A}(t) = \lim_{n \to \infty} \frac 1n \vc{V}^{A}(nt), \quad \ol{\vc{Z}}^{A}(t) = \lim_{n \to \infty} \frac 1n \vc{Z}^{A}(nt), \ol{\vc{Y}}^{A}(t) = \lim_{n \to \infty} \frac 1n \vc{Y}^{A}(nt), \quad t \ge 0,
\end{align*}
as long as they exists. Since $J(t)$ is a finite state Markov chain having the stationary distribution, the law of large numbers yields that 
\begin{align}
\label{eq:drift 1}
  \ol{V}^{A}_{k}(t) = \lim_{n \to \infty} \frac 1n \int_{0}^{nt} v_{k}(J(s)) ds = t \ol{v}_{k}, \qquad k \in A, \qquad w.p.1.
\end{align}
Similar to the reflection mapping for \eq{Z 2} and \eq{ZY 1}, it follows from \eq{ZA 0}, the fluid limit of $\ol{\vc{Z}}^{A}(\cdot)$ of $\vc{Z}^{A}(\cdot)$ is characterized by
\begin{align}
\label{eq:ZA 1}
 & \ol{\vc{Z}}^{(A)}_{A}(t) = \ol{\vc{Z}}^{(A)}_{A}(0) + t \ol{\vc{v}}_{A} + R_{A} \ol{\vc{Y}}^{(A)}_{A}(t) \ge \vc{0}_{A}, \qquad t \ge 0,\\
\label{eq:ZA 2}
 & \ol{\vc{Z}}^{(A)}_{A^{c}}(t) = \ol{\vc{Z}}^{(A)}_{A^{c}}(0) + t \ol{\vc{v}}_{A^{c}} - P^{\rs{t}}_{A^{c},A} \ol{\vc{Y}}^{(A)}_{A}(t), \qquad t \ge 0,\\
\label{eq:ZA 3}
 & \int_{0}^{t} \ol{Z}^{(A)}_{k}(s) \ol{Y}^{(A)}_{k}(ds) = 0, \quad \ol{Y}^{(A)}_{k}(0) = 0, \quad \mbox{$\ol{Y}^{(A)}_{k}(t) \uparrow$ in $t$}, \quad k \in A, t \ge 0.
\end{align}
We first note that \eq{ZA limit 1} is obtained if we show that $\lim_{t \to \infty} \ol{\vc{Z}}^{(A)}_{k}(t) = -\infty$ w.p.1 for some $k \in A^{c}$. It is immediate from \eq{ZA 2} that this holds if $\ol{v}_{k} < 0$ for some $k \in A^{c}$. In particular, this is the case for $A = \emptyset$, that is, $A^{c} =\sr{K}$ by \eq{stability 1}. Thus, we can assume that $A \ne \emptyset, \sr{K}$. That is, both of $A$ and $A^{c}$ are not empty. From \eq{ZA 1}, we have
\begin{align*}
  - \ol{\vc{Y}}^{(A)}_{A}(t) \le R_{A}^{-1} (t\ol{\vc{v}}_{A} + \ol{\vc{Z}}^{(A)}_{A}(0)).
\end{align*}
Substituting this into \eq{ZA 2}, we have
\begin{align}
\label{eq:ZA 4}
  \ol{\vc{Z}}^{(A)}_{A^{c}}(t) = \ol{\vc{Z}}^{(A)}_{A^{c}}(0) + t (\ol{\vc{v}}_{A^{c}} + P^{\rs{t}}_{A^{c},A} R_{A}^{-1} \ol{\vc{v}}_{A}) + R_{A}^{-1} \ol{\vc{Z}}^{(A)}_{A}(0).
\end{align}
Since $R, R_{A}, R_{A^{c}}$ are invertible, we have, by Schur's formula (e.g., see 0.7.3.1 of \cite{HornJohn2013} and page 245 of \cite{ChenZhan2000b}),
\begin{align*}
  R^{-1} = \left(\begin{array}{cc} (R_{A} - R_{A,A^{c}} R_{A^{c}}^{-1} R_{A^{c},A})^{-1} & -(R_{A} - R_{A,A^{c}} R_{A^{c}}^{-1} R_{A^{c},A})^{-1} R_{A,A^{c}} R_{A^{c}}^{-1} \\
  -(R_{A^{c}} - R_{A^{c},A} R_{A}^{-1} R_{A,A^{c}})^{-1} R_{A^{c},A} R_{A}^{-1} & (R_{A^{c}} - R_{A^{c},A} R_{A}^{-1} R_{A,A^{c}})^{-1} \end{array}\right).
\end{align*}
Since $R_{A^{c},A} = - P^{\rs{t}}_{A^{c},A}$, it follows from the $A_{c}$- column vector of $R^{-1} \ol{\vc{v}} < \vc{0}$ that
\begin{align*}
  (R_{A^{c}} + P^{\rs{t}}_{A^{c},A} R_{A}^{-1} R_{A,A^{c}})^{-1} (P^{\rs{t}}_{A^{c},A} R_{A}^{-1} \ol{\vc{v}}_{A} + \ol{\vc{v}}_{A^{c}}) < \vc{0}_{A^{c}}.
\end{align*}
Since $(R_{A^{c}} + P^{\rs{t}}_{A^{c},A} R_{A}^{-1} R_{A,A^{c}})^{-1}$ is a nonnegative matrix, there is at least one $k \in A_{c}$ such that $[P^{\rs{t}}_{A^{c},A} R_{A}^{-1} \ol{\vc{v}}_{A}]_{k} + \ol{\vc{v}}_{k} < 0$. Hence, \eq{ZA 4} implies that $\ol{\vc{Z}}^{(A)}_{k}(t) \to -\infty$ for this $k$.

\subsection{Proof of \lem{reachable 1}}
\label{sec:reachable 1}

From the statibility condition \eq{stability 2}, there exists an $i \in \sr{J}$ such that
\begin{align*}
  \lambda_{k}(i) + \sum_{\ell \in \sr{K}} \alpha_{\ell}(i) p_{\ell,k} = \alpha_{k}(i) < \mu_{k}(i).
\end{align*}
Since $\alpha_{k}(i)$ is the maximal in-flow rate at station $k$ under the background state $i$, $Z_{k}(t)$ decreases when $J(t) = i$. The sojourn time at state $i$ has the exponential distribution, so can be arbitrarily long. Hence, $Z_{k}(t)$ eventually hits $0$ with positive probability. This proves \eq{reachable 1} if $\vc{x} = (\vc{z},i)$ for any $\vc{z} \in S$. If $\vc{x} \not= (\vc{z},i)$, then $\vc{X}(t)$ arrives at $\vc{x}' = (\vc{z}',i)$ for some $\vc{z}' \in S$ in any finite time with positive probability by the irreducibility of the background Markov chain $J(\cdot)$. This proves (a).\\
To prove (b), assume that $\nu(S_{k} \times \sr{J}) = 0$ for some $k$. Since any sample path of $\vc{X}(t)$ starting from any $\vc{x}$ in $S$ hits $S_{k} \times \sr{J}$ with positive probability, $\vc{X}(t)$ must be transient, which contradicts that the stationary distribution exists.

\subsection{Proof of \lem{W 1}}
\label{sec:W 1}

On the time interval $(\tau^{\ex}_{A}(1), \tau^{\re}_{A}(1)]$, the sample path $\vc{Z}(t)$ is identical with $\vc{Z}^{(A^{c})}(t)$ because they are away from the boundary $S_{A}$. Hence,
\begin{align}
\label{eq:W 2}
  W_{A}(\tau_{x}) = \dd{E}_{\nu^{\ex}_{A}} \left(\left. \int_{\tau_{x}}^{\tau^{\re}_{A}} 1(\vc{Z}^{(A^{c})}(u) \in C(x)) du \right|\sr{F}_{\tau_{x}} \right).
\end{align}
For $t > 0$, let
\begin{align*}
  \sigma_{x}(t) = \inf\{u \ge t + \tau_{x}; \vc{Z}^{(A^{c})}(u) \in C(x)\},
\end{align*}
where $\sigma_{x}(t)$ may be larger than $\tau^{\re}_{A}(1)$. Since $Z^{(A^{c})}_{k}(t) \to -\infty$ almost surely for some $k \in A^{c}$ as $t \to \infty$ by \lem{ZA limit 1}, $\vc{Z}^{(A^{c})}(t)$ eventually leaves $S \equiv \dd{R}_{+}^{d}$ as $t \to \infty$, we have that $\int_{\tau_{x}}^{\infty} 1(\vc{Z}^{(A^{c})}(u) \in C(x)) du < \infty$ w.p.1. Hence, on $\{\tau_{x} < \infty\}$, $t \le \int_{\tau_{x}}^{\infty} 1(\vc{Z}^{(A^{c})}(u) \in C(x)) du$ implies that $t + \tau_{x} \le \sigma_{x}(t) < \infty$. Thus, \eq{W 2} yields
\begin{align}
\label{eq:W 3}
  W_{A}(\tau_{x}) & = \dd{E}_{\nu^{\ex}_{A}} \left(\left. \int_{0}^{\infty} 1\left(\int_{\tau_{x}}^{\tau^{\re}_{A}} 1(\vc{Z}^{(A^{c})}(u) \in C(x)) du > t\right) dt \right|\sr{F}_{\tau_{x}} \right) \nonumber\\
 & \le \dd{E}_{\nu^{\ex}_{A}} \left(\left. \int_{0}^{\infty} 1\left(\int_{\tau_{x}}^{\infty} 1(\vc{Z}^{(A^{c})}(u) \in C(x)) du > t\right) dt \right|\sr{F}_{\tau_{x}} \right) \nonumber\\
  & \le \dd{E}_{\nu^{\ex}_{A}} \left(\left. \int_{0}^{\infty} 1\left(\sigma_{x}(t) < \infty \right) dt \right|\sr{F}_{\tau_{x}} \right) \nonumber\\
  & \le \int_{0}^{\infty} \dd{P}_{\nu^{\ex}_{A}} \left(\left. \sigma_{x}(t) < \infty \right|\sr{F}_{\tau_{x}} \right) dt.
\end{align}
We next apply change of measure of $\dd{P}_{\nu^{\ex}_{A}}$ by $E^{\vc{\theta}}_{A^{c}}(t)$ of \eq{EM 2} for $\vc{\theta} \in \Gamma^{-}_{A^{c}}$, then
\begin{align*}
 & \dd{P}_{\nu^{\ex}_{A}} \left(\left. \sigma_{x}(t) < \infty \right|\sr{F}_{\tau_{x}} \right) \nonumber\\
 & = \widetilde{E}_{\nu^{\ex}_{A}}\left( \frac {e^{\br{\vc{\theta},\vc{Z}^{(A^{c})}(\tau_{x})}} h_{\vc{\theta}}(J(\tau_{x}))} {e^{\br{\vc{\theta},\vc{Z}^{(A^{c})}(\sigma_{x}(t))}} h_{\vc{\theta}}(J(\sigma_{x}(t)))} 1\left(\sigma_{x}(t) < \infty \right) e^{\gamma(\vc{\theta}) (\sigma_{x}(t) - \tau_{x}) + \sum_{k \in A^{c}} \gamma_{k}(\vc{\theta}) Y^{(A^{c})}_{k}((\tau_{x},\sigma_{x}(t)])}\right).
\end{align*}
Since $h_{\vc{\theta}}(i)$ is bounded and bounded away from $0$ for $i \in \sr{J}$, $h_{\vc{\theta}}(J(\tau_{x}))/h_{\vc{\theta}}(J(\sigma_{x}(t)))$ is bounded. Since $e^{\br{\vc{\theta},\vc{Z}^{(A^{c})}(\tau_{x})}- \br{\vc{\theta},\vc{Z}^{(A^{c})}(\sigma_{x}(t))}}$ is bounded by \eq{C 1}. Hence, $t + \tau_{x} \le \sigma_{x}(t)$ and $\vc{\theta} \in \Gamma^{-}_{A}$ implies that the right-hand side of the above equation is bounded by $c_{0} e^{\gamma(\vc{\theta})t}$ for some positive constant $c_{0}$. Since $\gamma(\vc{\theta}) < 0$ for $\vc{\theta} \in \Gamma^{-}_{A}$, \eq{W 3} yields the upper bound of \eq{W x 1}. The lower bound is obvious because
\begin{align}
\label{eq:W 4}
   \lim_{x \to \infty} W_{\tau_{x}} = \dd{E}_{\nu} \left( \int_{0}^{\infty} 1(\vc{Z}^{(\emptyset)}(u) \in C(x)) du \right) > 0.
\end{align}

\section{Concluding remarks} \setnewcounter
\label{sec:concluding}

In this paper, we have shown that, for $d=2$, the tail asymptotic problems of the MMFN process can be solved essentially in the same way as the two-dimensional SRBM studied in \cite{DaiMiya2011,DaiMiya2013}. We extend this idea for $d \ge 3$, reducing the problem to solve the fixed point equation \eq{fixed 1} for $(d-1)$-dimensional sets. However, there remain many problems about the solution of this fixed point equation.
\begin{itemize}
\item [(i)] By \lem{fixed 1}, there is the maximal solution $\{D^{(\max)}_{k}; k \in \sr{K}\}$ of \eq{fixed 1}, but the uniqueness of the nontrivial solution of \eq{fixed 1} and $D^{(\max)} = \sr{D}$ are open problems.
\item [(ii)] What is concrete shape of $D^{(\max)}$ for the solution ? We guess that there must be an algorithm to draw its boundary face.
\item [(iii)] We only have upper bounds \eq{upper 2} for $d \ge 3$. Are they identical with the logarithmic decay rates\hspace{0.5mm}?
\item [(iv)] Can finer tail asymptotics such as the so-called exact tail asymptotic be obtained in the framework of \lem{marginal 2}\hspace{0.5mm}?
\end{itemize}

These are interesting but hard problems.

Another issue is about the approach based on change of measure. This work originally started with this approach. However, we found that it does not so much earn the tail asymptotic results in the present framework. Nevertheless, we believe the approach is promising to get finer results, although there are many problems to be overcome.

In this paper, we study a kind of simple Markov modulated queueing networks. We hope this study is the first step to analyze the asymptotic behaviors of Markov modulated stochastic networks for more general input processes such as renewal on-off sources and nonnegative jumps at the transition instants of the background process. Further generalization would be more general net flow processes such as Markov modulated L\'{e}vy processes with nonnegative jumps. 

\section*{Acknowledgements}

I am grateful to the referees for their helpful comments and suggestions on this paper. In particular, the proof of the stability is reconsidered, which was based on a fluid approximation and $\psi$-irreducibility in the original submission.

\def\cprime{$'$} \def\cprime{$'$} \def\cprime{$'$} \def\cprime{$'$}
  \def\cprime{$'$} \def\cprime{$'$} \def\cprime{$'$}

%\bibliography{../../../texmf/bib/dai20190702}

\begin{thebibliography}{39}
\expandafter\ifx\csname natexlab\endcsname\relax\def\natexlab#1{#1}\fi
\expandafter\ifx\csname url\endcsname\relax
  \def\url#1{\texttt{#1}}\fi
\expandafter\ifx\csname urlprefix\endcsname\relax\def\urlprefix{URL }\fi
\providecommand{\eprint}[2][]{\url{#2}}

\bibitem[{Adan et~al.(2009)Adan, Mandjes, Scheinhardt and
  Tzenova}]{AdanMandScheTzen2009}
\textsc{Adan, I.}, \textsc{Mandjes, M.}, \textsc{Scheinhardt, W.} and
  \textsc{Tzenova, E.} (2009).
\newblock On a generic class of two-node queueing systems.
\newblock \textit{Queueing Systems}, \textbf{61} 37--63.

\bibitem[{Asmussen(2003)}]{Asmu2003}
\textsc{Asmussen, S.} (2003).
\newblock \textit{Applied probability and queues}, vol.~51 of
  \textit{Applications of Mathematics (New York)}.
\newblock 2nd ed. Springer-Verlag, New York.
\newblock Stochastic Modelling and Applied Probability.

\bibitem[{Atar and Dupuis(1999)}]{AtarDupu1999}
\textsc{Atar, R.} and \textsc{Dupuis, P.} (1999).
\newblock Large deviations and queueing networks: methods for rate function
  identification.
\newblock \textit{Stochastic Processes and their Applications}, \textbf{84}.

\bibitem[{Avram et~al.(2001)Avram, Dai and Hasenbein}]{AvraDaiHase2001}
\textsc{Avram, F.}, \textsc{Dai, J.~G.} and \textsc{Hasenbein, J.~J.} (2001).
\newblock Explicit solutions for variational problems in the quadrant.
\newblock \textit{Queueing Systems}, \textbf{37} 259--289.

\bibitem[{Baccelli and Br{\'e}maud(2003)}]{BaccBrem2003}
\textsc{Baccelli, F.} and \textsc{Br{\'e}maud, P.} (2003).
\newblock \textit{Elements of queueing theory: Palm martingale calculus and
  stochastic recurrences}, vol.~26 of \textit{Applications of Mathematics}.
\newblock 2nd ed. Springer, Berlin.

\bibitem[{Borovkov and Mogul{\cprime}ski{\u\i}(2001)}]{BoroMogu2001}
\textsc{Borovkov, A.~A.} and \textsc{Mogul{\cprime}ski{\u\i}, A.~A.} (2001).
\newblock Large deviations for {Markov} chains in the positive quadrant.
\newblock \textit{Russian Mathematical Surveys}, \textbf{56} 803--916.

\bibitem[{Chen and Mandelbaum(1991)}]{ChenMand1991}
\textsc{Chen, H.} and \textsc{Mandelbaum, A.} (1991).
\newblock Discrete flow networks: bottleneck analysis and fluid approximations.
\newblock \textit{Mathematics of Operations Research}, \textbf{16} 408--446.

\bibitem[{Chen and Yao(2001)}]{ChenYao2001}
\textsc{Chen, H.} and \textsc{Yao, D.} (2001).
\newblock \textit{Fundamentals of Queueing Networks: Performance, Asymptotics,
  and Optimization}.
\newblock Springer-Verlag, New York.

\bibitem[{Chen and Zhang(2000)}]{ChenZhan2000b}
\textsc{Chen, H.} and \textsc{Zhang, H.} (2000).
\newblock A sufficient condition and a necessary condition for the diffusion
  approximations of multiclass queueing networks under priority service
  disciplines.
\newblock \textit{Queueing Systems}, \textbf{34} 237--268.

\bibitem[{Cohen(1981)}]{Cohe1981}
\textsc{Cohen, J.~E.} (1981).
\newblock Convexity of the dominant eigenvalue of an essentially nonnegative
  matrix.
\newblock \textit{Proceedings of the American Mathematical Society},
  \textbf{81} 657--658.

\bibitem[{Dai(1995)}]{Dai1995}
\textsc{Dai, J.~G.} (1995).
\newblock On positive {H}arris recurrence of multiclass queueing networks: a
  unified approach via fluid limit models.
\newblock \textit{Annals of Applied Probability}, \textbf{5} 49--77.

\bibitem[{Dai and Miyazawa(2011)}]{DaiMiya2011}
\textsc{Dai, J.~G.} and \textsc{Miyazawa, M.} (2011).
\newblock Reflecting {Brownian} motion in two dimensions: Exact asymptotics for
  the stationary distribution.
\newblock \textit{Stochastic Systems}, \textbf{1} 146--208.

\bibitem[{Dai and Miyazawa(2013)}]{DaiMiya2013}
\textsc{Dai, J.~G.} and \textsc{Miyazawa, M.} (2013).
\newblock Stationary distribution of a two-dimensional {SRBM}: geometric views
  and boundary measures.
\newblock \textit{Queueing Systems}, \textbf{74} 181--217.

\bibitem[{Debicki et~al.(2007)Debicki, Dieker and Rolski}]{DebiDiekRols2007}
\textsc{Debicki, K.}, \textsc{Dieker, A.~B.} and \textsc{Rolski, T.} (2007).
\newblock Quasi-product forms for {L\'evy}-driven fluid networks.
\newblock \textit{Mathematics of Operations Research}, \textbf{32} 629--647.

\bibitem[{Dupuis and Ellis(1997)}]{DupuElli1997}
\textsc{Dupuis, P.} and \textsc{Ellis, R.~S.} (1997).
\newblock \textit{A weak convergence approach to the theory of large
  deviations}.
\newblock Wiley Series in Probability and Statistics: Probability and
  Statistics, John Wiley \& Sons Inc., New York.
\newblock A Wiley-Interscience Publication.

\bibitem[{Dupuis et~al.(1991)Dupuis, Ellis and Weiss}]{DupuElliWeis1991}
\textsc{Dupuis, P.}, \textsc{Ellis, R.~S.} and \textsc{Weiss, A.} (1991).
\newblock Large deviations for {Markov} processes with discontinuous
  statistics, i: general upper bounds.
\newblock \textit{Annals of Probability}, \textbf{19} 1280--1297.

\bibitem[{Ethier and Kurtz(1986)}]{EthiKurt1986}
\textsc{Ethier, S.~N.} and \textsc{Kurtz, T.~G.} (1986).
\newblock \textit{{M}arkov Processes: Characterization and Convergence}.
\newblock Wiley, New York.

\bibitem[{Freidlin and Wentzell(1998)}]{FreiWent1998}
\textsc{Freidlin, M.~I.} and \textsc{Wentzell, A.~D.} (1998).
\newblock \textit{Random Perturbations of Dynamical Systems}.
\newblock 2nd ed. A Series of Comprehensive Studies in Mathematics, Springer.

\bibitem[{Horn and Johnson(2013)}]{HornJohn2013}
\textsc{Horn, R.~A.} and \textsc{Johnson, C.~R.} (2013).
\newblock \textit{Matrix Analysis}.
\newblock Cambridge University Press, New York.

\bibitem[{Ignatiouk-Robert(2000)}]{Igna2000}
\textsc{Ignatiouk-Robert, I.} (2000).
\newblock Large deviations of {{Jackson}} networks.
\newblock \textit{Annals of Applied Probability}, \textbf{10} 962--1001.

\bibitem[{Ignatiouk-Robert(2005)}]{Igna2005}
\textsc{Ignatiouk-Robert, I.} (2005).
\newblock Large deviations for processes with discontinuous statistics.
\newblock \textit{Annals of Probability}, \textbf{33} 1479--1508.

\bibitem[{Kallenberg(2001)}]{Kall2001}
\textsc{Kallenberg, O.} (2001).
\newblock \textit{Foundations of Modern Probability}.
\newblock 2nd ed. Springer Series in Statistics, Probability and its
  applications, Springer, {New York}.

\bibitem[{Kella(2001)}]{Kell2001}
\textsc{Kella, O.} (2001).
\newblock Markov-modulated feedforward fluid networks.
\newblock \textit{Queueing Systems}, \textbf{37} 141--161.

\bibitem[{Kella and Whitt(1996)}]{KellWhit1996}
\textsc{Kella, O.} and \textsc{Whitt, W.} (1996).
\newblock Stability and structural properties of stochastic storage networks.
\newblock \textit{Journal of Applied Probability}, \textbf{33} 1169--1180.

\bibitem[{Kunita and Watanabe(1963)}]{KuniWata1963}
\textsc{Kunita, H.} and \textsc{Watanabe, T.} (1963).
\newblock Notes on transformations of markov processes connected with
  multiplicative functionals.
\newblock \textit{Memoirs of the Faculty of Science, Kyushu}, \textbf{17}
  181--191.

\bibitem[{Majewski(2009)}]{Maje2009}
\textsc{Majewski, K.} (2009).
\newblock Functional continuity and large deviations for the behavior of
  single-class queueing networks.
\newblock \textit{Queueing Systems}, \textbf{61} 203--241.

\bibitem[{Meyer(2015)}]{Meye2015}
\textsc{Meyer, C.~D.} (2015).
\newblock Continuity of the perron root.
\newblock \textit{Linear and Multilinear Algebra}, \textbf{63} 1332--1336.

\bibitem[{Meyn and Down(1994)}]{MeynDown1994}
\textsc{Meyn, S.~P.} and \textsc{Down, D.} (1994).
\newblock Stability of generalized {J}ackson networks.
\newblock \textit{Annals of Applied Probability}, \textbf{4} 124--148.

\bibitem[{Meyn and Tweedie(1993)}]{MeynTwee1993}
\textsc{Meyn, S.~P.} and \textsc{Tweedie, R.~L.} (1993).
\newblock Generalized resolvents and {H}arris recurrence of {M}arkov processes.
\newblock In \textit{50 years after {D}oeblin: Developments on the Theory of
  {M}arkov Chains, {M}arkov Processes and Sums of Random Variables} (H.~Cohn,
  ed.). Contemporary Mathematics, Amer.\ Math.\ Soc., Providence, RI.

\bibitem[{Miyazawa(1994)}]{Miya1994}
\textsc{Miyazawa, M.} (1994).
\newblock Rate conservation laws: a survey.
\newblock \textit{Queueing Systems}, \textbf{15} 1--58.

\bibitem[{Miyazawa(2011)}]{Miya2011}
\textsc{Miyazawa, M.} (2011).
\newblock Light tail asymptotics in multidimensional reflecting processes for
  queueing networks.
\newblock \textit{TOP}, \textbf{19} 233--299.

\bibitem[{Miyazawa(2015)}]{Miya2015a}
\textsc{Miyazawa, M.} (2015).
\newblock A superharmonic vector for a nonnegative matrix with {QBD} block
  structure and its application to a {Markov} modulated two dimensional
  reflecting process.
\newblock \textit{Queueing Systems}, \textbf{81} 1--48.

\bibitem[{Miyazawa(2017)}]{Miya2017a}
\textsc{Miyazawa, M.} (2017).
\newblock Martingale approach for tail asymptotic problems in the generalized
  {Jackson} network.
\newblock \textit{Probability and Mathematical Statistics}, \textbf{37}
  395--430.

\bibitem[{Nussbaum(1986)}]{Nuss1986}
\textsc{Nussbaum, D.~R.} (1986).
\newblock Convexity and log convexity for the spectral radius.
\newblock \textit{Linear Algebra and its Applications}, \textbf{73} 59--122.

\bibitem[{Ozawa and Kobayashi(2018)}]{OzawKoba2018}
\textsc{Ozawa, T.} and \textsc{Kobayashi, M.} (2018).
\newblock Exact asymptotic formulae of the stationary distribution of a
  discrete-time two-dimensional {QBD} process.
\newblock \textit{Queueing Systems}, \textbf{90} 351--403.

\bibitem[{Palmowski and Rolski(2002)}]{PalmRols2002}
\textsc{Palmowski, Z.} and \textsc{Rolski, T.} (2002).
\newblock A technique of the exponential change of measure for {Markov}
  processes.
\newblock \textit{Bernoulli}, \textbf{8} 767--785.

\bibitem[{Rudin(1976)}]{Rudi1976}
\textsc{Rudin, W.} (1976).
\newblock \textit{Principles of Mathematical Analysis}.
\newblock 3rd ed. International Series in Pure and Applied Mathematics,
  McGraw-Hill.

\bibitem[{Seneta(1981)}]{Sene1981}
\textsc{Seneta, E.} (1981).
\newblock \textit{Non-negative matrices and {Markov} chains}.
\newblock 2nd ed. Springer Series in Statistics, Springer.

\bibitem[{Whitt(2002)}]{Whit2002}
\textsc{Whitt, W.} (2002).
\newblock \textit{Stochastic-process limits}.
\newblock Springer, New York.

\end{thebibliography}

\end{document}